\documentclass[11pt,reqno]{amsart}
\usepackage{amssymb}
\usepackage{amscd}
\usepackage[latin1]{inputenc}
\usepackage{amsmath}
\usepackage{amsfonts}
\usepackage{faktor}
\usepackage{latexsym}
\usepackage{verbatim}
\usepackage{graphicx}
\usepackage{epsfig}

\usepackage{enumitem}
\usepackage[linktocpage=true]{hyperref}
\usepackage[capitalize]{cleveref}
\hypersetup{citecolor = black,colorlinks,linkcolor = black,urlcolor = black}

\newtheorem{theorem}{Theorem}[section]

\newtheorem{lemma}[theorem]{Lemma}
\newtheorem{question}[theorem]{Question}
\newtheorem{proposition}[theorem]{Proposition}
\newtheorem{corollary}[theorem]{Corollary}

\usepackage[hmargin=3cm,vmargin=3cm]{geometry}

\theoremstyle{definition}
\newtheorem{definition}[theorem]{Definition}
\newtheorem{example}[theorem]{Example}
\newtheorem{remark}[theorem]{Remark}

\theoremstyle{plain}
\newcounter{MainTheoremCounter}

\newtheorem{Maintheorem}[MainTheoremCounter]{Theorem}

\newcommand{\A}{\mathcal{A}}
\newcommand{\OO}{\mathrm{Orb}}
\newcommand{\Z}{\mathbb{Z}}
\newcommand{\N}{\mathbb{N}}
\newcommand{\R}{\mathbb{R}}
\newcommand{\T}{\mathbb{T}}

\begin{document}

\title[On subshifts with low maximal pattern complexity]{On subshifts with low maximal pattern complexity}

\author{Anh N. Le}
\address{Anh N. Le\\
Department of Mathematics\\
University of Denver\\
2390 S. York St.\\
Denver, CO 80208}
\email{anh.n.le@du.edu}
%\urladdr{http://www.math.du.edu/$\sim$rpavlov/}

\author{Ronnie Pavlov}
\address{Ronnie Pavlov\\
Department of Mathematics\\
University of Denver\\
2390 S. York St.\\
Denver, CO 80208}
\email{rpavlov@du.edu}
\urladdr{http://www.math.du.edu/$\sim$rpavlov/}

\author{Casey Schlortt}
\address{Casey Schlortt\\
Department of Mathematics\\
University of Denver\\
2390 S. York St.\\
Denver, CO 80208}
\email{casey.schlortt@du.edu}

\thanks{The second author gratefully acknowledges the support of a Simons Foundation Collaboration Grant.}

\keywords{subshifts, maximal pattern complexity, topological groups}
\renewcommand{\subjclassname}{MSC 2020}
\subjclass[2020]{Primary: 37B10; Secondary: 37B05}

\begin{abstract}
For a finite alphabet $\mathcal{A}$ and a sequence $x \in \A^{\mathbb{N}}$, Kamae and Zamboni defined the maximal pattern complexity function $p^*_x(n)$ as a natural generalization of usual word complexity. They defined a nonperiodic sequence $x$ to be pattern Sturmian if it achieves the minimal growth rate $p^*_x(n) = 2n$, and asked the question of whether one could classify recurrent pattern Sturmian sequences. 

We answer their question by characterizing recurrent pattern Sturmian sequences as one of two known types: either a coding of an irrational circle rotation by two intervals, or an element of what we call a nearly simple Toeplitz subshift.

We also show that nonrecurrent pattern Sturmian sequences are either very close to constant (such examples were given by Kamae and Zamboni) or a (nonrecurrent) coding of an irrational circle rotation by two intervals.

Our main new technique is to use topological properties of the maximal equicontinuous factor (MEF) of the subshift generated by $x$. In this way, we prove a general structural result about sequences with non-superlinear maximal pattern complexity: they are either nonrecurrent or minimal with MEF either an odometer or the product of a circle with a finite cyclic group.

\end{abstract}

\maketitle

\setcounter{tocdepth}{2}
\tableofcontents

\section{Introduction}

%[[Decide on terminology: pattern-Sturmian sequences, or pattern-Sturmian words, or pattern-Sturmian subshifts]]

%[[Decide $\Z$ action or $\N$ action]] [[May be $\N$ action?]]

%[[Since we deal only with $\{0, 1\}$-subshift in this paper, drop $\{0, 1\}$.]]

%[[``Non-superlinear'' sounds clunkly. Can we call it ``linear''?]]

One of the most fundamental invariants in the study of dynamical systems is that of entropy. Informally, topological entropy measures the amount of chaoticity/unpredictability displayed by orbits of points in the system, by counting growth rates of orbits of length $n$ distinguishable at different scales; exponential growth corresponds to positive entropy.

However, there are a variety of important systems which have zero entropy, such as interval exchanges, many billiards, and rank one systems, and so one can use finer information to distinguish between such systems. This is sometimes a quite technical task due to dependence on scale, but for symbolically defined systems, it is particularly simple and reduces to a single function called word complexity. Specifically, given a sequence $x \in \mathcal{A}^{\N_0}$, for $\N_0 = \N \cup \{0\}$, over an alphabet $\mathcal{A}$, the \emph{word complexity} function $p_x$ is defined by $p_x(n)$ equal to the number of $n$-letter strings/words appearing within $x$. For a subshift $X$, one instead takes the union over all $x \in X$. For instance, if $X$ consists of all $\{0,1\}$-sequences without consecutive $1$s, then $p_X(3) = 5$, since the $3$-letter strings appearing in $X$ are $000, 001, 010, 100, 101$. 

The classical Morse-Hedlund theorem (see \cite{MorseHedlund}) states that $x$ is not eventually periodic if and only if $p_x(n) \geq n+1$ for all $n$. Since eventually periodic $x$ have bounded word complexity, this shows that $n+1$ is the slowest nontrivial growth rate for $p_x(n)$. Interestingly, this minimum growth rate does occur, and is equivalent to $x$ being a so-called \emph{Sturmian sequence} defined via codings of irrational circle rotations (see \Cref{sec:defs} for more details). Several recent works (\cite{CK2,CK1,DDMP1,DDMP2}) have focused on implications of linear growth of $p_x$, and \cite{Creutz_Pavlov-less_than_3/2, Creutz_Pavlov-low_complexity} have demonstrated that $p_x(n) \approx 1.5n$ is a threshold for several dynamical behaviors.

Continuing this thread of inquiry, in \cite{Kamae-Zamboni-sequence_entropy}, Kamae and Zamboni defined a  finer measure called maximal pattern complexity. For any $n$, the \emph{maximal pattern complexity} function $p^*_x$ is defined by $p^*_x(n)$ equals to the maximum, over all sets 
$\tau \subseteq \mathbb{Z}$ with $|\tau| = n$, of the number of patterns in $\mathcal{A}^{\tau}$ appearing within $x$. For instance, if $x = 010110 \ldots$, then $p^*_x(2) = 4$, because all four patterns $0\_0$, $0\_1$, $1\_0$, $1\_1$ appear in $x$.

%Interestingly, there are some natural parallels between word complexity and maximal pattern complexity. 
A main result of \cite{Kamae-Zamboni-sequence_entropy} was an analog of the Morse-Hedlund theorem for maximal pattern complexity, stating that $x$ is not eventually periodic if and only if $p^*_x(n)\geq 2n$ for all $n$. For natural reasons, they defined sequences achieving the minimal growth rate of $p^*_x(n) = 2n$ as \emph{pattern Sturmian sequences}.

In \cite{Kamae-Zamboni-discrete, Kamae-Zamboni-sequence_entropy}, Kamae and Zamboni provided three classes of examples of pattern Sturmian sequences, which we call \emph{simple circle rotation coding sequences} (this includes all Sturmian sequences), \emph{simple Toeplitz sequences}, and \emph{almost constant sequences} (see \Cref{sec:defs} for definitions). In a later paper, Kamae and co-authors (\cite{Gjini_Kamae_Bo_Yu-Mei-toeplitz}) examined the Toeplitz case in more detail, and showed that a slightly larger class, which we call \emph{nearly simple Toeplitz sequences}, are also pattern Sturmian. 
Nearly simple Toeplitz sequences are recurrent (in fact uniformly recurrent), simple circle rotation coding sequences can be either recurrent or nonrecurrent, and almost constant sequences are highly nonrecurrent. In the paper \cite{Kamae-Zamboni-sequence_entropy}, Kamae and Zamboni asked the following question:

\begin{question}[{\cite[Problem 1]{Kamae-Zamboni-sequence_entropy}}]\label{ques:Kamae-Zamboni}
What is the general structure of recurrent pattern Sturmian sequences?
\end{question}

Our main results are a complete characterization of recurrent pattern Sturmian sequences (answering \Cref{ques:Kamae-Zamboni}) and a near characterization of nonrecurrent pattern Sturmian sequences, which together show that all examples are essentially of one of the three known types. %In particular, we prove the following.
%\Anote{Note here so we don't forget. The previous sentence is no longer accurate. We have sequence arising from orbit closure of nearly simple Toeplitz sequences.}

\begin{Maintheorem}\label{mainthm:sturmian}
If $x \in \{0,1\}^{\N_0}$ is recurrent, then $x$ is pattern Sturmian if and only if it is either a recurrent simple circle rotation coding sequence or a sequence in a nearly simple Toeplitz subshift.
%\begin{enumerate}
%    \item either a coding of an irrational circle rotation by two intervals,
%    \item or an almost simple Toeplitz.
%\end{enumerate}
\end{Maintheorem}

\begin{Maintheorem}\label{mainthm:sturmian-nonrecurrent}
If $x \in \{0,1\}^{\N_0}$ is nonrecurrent and pattern Sturmian, then it is either a nonrecurrent simple circle rotation coding sequence or almost constant.
\end{Maintheorem}

The latter is only a near characterization since not all almost constant sequences are pattern Sturmian, and a main remaining question on this topic is to find a description of those which are; see \Cref{q:almcon}.

Another context in which pattern Sturmian subshifts have been important is in the study of the spectrum of Schr\"{o}dinger operators. A general heuristic in this area is that two-sided sequences of sufficiently `low complexity' should have Schr\"{o}dinger operators with spectrum of zero Lebesgue measure and all spectral measures purely singular continuous. This structure has been proved for Sturmian sequences (\cite{schrod1, schrod2}) and some simple Toeplitz sequences (\cite{schrod3}). In unpublished work (\cite{Damanik_Liu_Qu-spectralProperties}), Damanik, Liu, and Qu gave partial results in the setting of pattern Sturmian sequences. They showed the desired structure for nearly simple Toeplitz sequences, gave an outline for the use of S-adic decomposition to approach more general simple circle rotation coding sequences, and proved that two-sided non-recurrent almost constant sequences cannot be pattern Sturmian. They say ``given that the class of all pattern Sturmian sequences is not yet fully understood [...] it is more or less hopeless to attack Conjectures 1.2 and 1.4 head-on.'' A `head-on approach' is now possible: our Theorems \ref{mainthm:sturmian} and \ref{mainthm:sturmian-nonrecurrent} imply that all one-sided pattern Sturmian sequences come from either circle coding, nearly simple Toeplitz, or almost constant sequences. Since \cite{Damanik_Liu_Qu-spectralProperties} showed that the almost constant case is impossible for two-sided sequences and resolved their conjectures for nearly simple Toeplitz, the only remaining case is simple circle rotation coding sequences.

Previous results on pattern Sturmian sequences have been proved mostly via purely combinatorial methods. The main idea behind the proof of \Cref{mainthm:sturmian} is the use of structural results about the dynamics of the subshift $X$ generated by $x$. Specifically, sequences with subexponential maximal pattern complexity are known to yield so-called null subshifts, and any minimal null subshift is an almost $1$-$1$ extension of a group rotation, known as its maximal equicontinuous factor (see \Cref{kerrli}). %\Anote{The paragraph in the beginning of next page says ``Similar examination of the maximal equicontinuous factors yields...'' but we have not mentioned MEF anywhere in the main text yet. So change ``extensions of group rotations'' to ``extensions of their maximal equicontinuous factors''}

%Interestingly, Kamae and Zamboni introduced their Toeplitz pattern Sturmian examples in a section of \cite{Kamae-Zamboni-discrete} titled `A class of recurrent pattern Sturmian words not arising from rotations'; viewing them as in fact arising from rotations of odometers is in some sense central to our approach.

Examination of these group rotations is central to our proofs, and in fact applies to the more general setting of non-superlinear  complexity: the sequence $x$ has \emph{non-superlinear maximal pattern complexity} if 
\[
    \liminf_{n \to \infty} \frac{p_x^*(n)}{n} < \infty.
\]
See \Cref{sec:defs} for definitions of other terms appearing in the next theorem.

\begin{Maintheorem}\label{mainthm:non-superlinear}
If $x \in \{0, 1\}^{\N_0}$ is recurrent, not periodic, and has non-superlinear maximal pattern complexity, then $x$ is uniformly recurrent and either
\begin{enumerate}
    \item a periodic interleaving of finitely many sequences which are each either a circle rotation interval coding sequence or constant, where all of the circle rotation interval coding sequences are associated with the same irrational rotation or
    \item an element of an $m$-hole Toeplitz subshift for some $m > 0$. 
\end{enumerate}
\end{Maintheorem}
%\Anote{Did we account for the case $x$ is periodic in Theorem C?}

A surprising open question in the theory of null systems is whether there exists a null sequence $x$ (meaning that $p^*_x(n)$ grows subexponentially) which is recurrent but not uniformly recurrent. One of the results used in our proof of \Cref{mainthm:sturmian} is the following, which shows that such sequences must have maximal pattern complexity growing on the order of $n \ln n$.

\begin{Maintheorem}\label{mainthm:recnotmin}
If $x$ is recurrent and not uniformly recurrent, then 
\[
\liminf_{n \to \infty} \frac{p^*_x(n)}{n \ln n} > 0.
\]
\end{Maintheorem}

\noindent \textbf{Acknowledgments.} We thank David Damanik for informing us about the connection between pattern Sturmian sequences and spectrum of Schr\"{o}dinger operators. 

\section{Definitions and preliminaries}\label{sec:defs}

\subsection{Basics on symbolic dynamics and maximal pattern complexity}

%\textcolor{red}{[I added a subsection. Currently, Section 2 only has one subsection which is not symmetric.]}

We use $\N$ to denote the set of natural numbers $\{1, 2, \ldots\}$ and $\N_0 = \N \cup \{0\}$. An \emph{alphabet} is a finite set. Even though all the definitions and most of our theorems would generalize to arbitrary finite alphabet $\mathcal{A}$, for ease of notation, we will restrict to $\mathcal{A} = \{0, 1\}$ in this paper.
All sequences considered in this paper are one-sided, i.e. in $\A^{\N_0}$, and throughout, the space $\A^{\N_0}$ is endowed with the product topology. 
%In this paper, $\A$ will be $\{0,1\}$ except for \Cref{mainthm:recnotmin}, which applies to arbitrary $\A$. 

\begin{definition}
    Let $\sigma: \A^{\N_0} \to \A^{\N_0}$ be the \emph{(left) shift map} defined by $(\sigma x)(n) = x(n+1)$ for $n \in \N_0$.
    The \emph{orbit} of a sequence $x$, denoted $\OO(x)$, is the set $\{\sigma^n(x)\}_{n \geq 0}$, and the 
    \emph{orbit closure} of $x$ is $\overline{\OO(x)}$.
\end{definition}    

\begin{definition}
     A \emph{subshift} $(X, \sigma)$ is defined by $X \subseteq \A^{\N_0}$ which is closed and satisfies $\sigma(X) \subseteq X$.
     For any sequence $x$, its orbit closure $\overline{\OO(x)}$ is a subshift.
\end{definition}

\begin{definition}
    A sequence $x$ is \emph{periodic} if there exists some $t \in \N$ such that $\sigma^t(x) = x$, or equivalently such that $x(n) = x(n+t)$ for all $n \in \N_0$. A sequence $x$ is \emph{eventually periodic} if there exists some $s \in \N$ such that $\sigma^s(x)$ is periodic, or equivalently there exist $s,t \in \N$ such that $x(n) = x(n+t)$ for all $n > s$.
\end{definition}

\begin{definition}
    A sequence $x$ is \emph{recurrent} if for all $L \in \N$ there exists some $M \in \N$ such that 
    $x(0) \ldots x(L-1) = x(M) \ldots x(M+L-1)$.  If for all $L \in \N$, the set 
    $R_L=\{M: x(0) \ldots x(L-1) = x(M) \ldots x(M+L-1)\}$ is syndetic (there are bounded gaps between subsequent elements), then $x$ is \emph{uniformly recurrent}.
\end{definition}

\begin{definition}
A subshift $X$ is \emph{minimal} if it does not properly contain any nonempty subshift. It is well-known that $x$ is uniformly recurrent if and only if $\overline{\OO(x)}$ is minimal.
\end{definition}

\begin{definition}
    A \emph{window} $\boldsymbol{\tau}$ is any finite subset of $\mathbb{N}_0$. The \emph{$\tau$-language} of a sequence $x$ is $L_x(\tau) := \{(\sigma^n x)(\tau)\}_{n \in \N_0} \subseteq \A^\tau$.
    The \emph{$\tau$-language} of a subshift $X$ is $L_X(\tau) := \bigcup_{x \in X} L_x(\tau)$. The \emph{maximal pattern complexity} of a sequence $x$ is defined by $p^*_x(n) = \max_{|\tau| = n} |L_x(\tau)|$.
    In the case $\tau = \{0,1, \ldots, n-1\}$, we represent $L_x(\tau)$ and $L_X(\tau)$ by $L_x(n)$ and $L_X(n)$ respectively.% Finally, $L(X) = \bigcup_\tau L_X(\tau)$ is the \emph{language} of $X$.
    \end{definition}

The following analogue of the Morse-Hedlund theorem (\cite{MorseHedlund}) was proven in \cite{Kamae-Zamboni-sequence_entropy}.

\begin{theorem}[\cite{Kamae-Zamboni-sequence_entropy}]\label{KZthm}
If $x$ is not eventually periodic, then $p^*_x(n) \geq 2n$ for all $n$.
\end{theorem}

In analogue with Sturmian sequences, which have minimum block complexity among nonperiodic sequences (for example, see \cite{fogg}), Kamae and Zamboni defined \emph{pattern Sturmian sequences} as those of minimum maximal pattern complexity. It is immediate that such sequences have alphabet with two letters, so we do not lose any generality in assuming 
$\A = \{0,1\}$ for such sequences. 

\begin{definition}
    A sequence $x \in \{0,1\}^{\N_0}$ is \emph{pattern Sturmian} if $p^*_x(n) = 2n$ for all $n$, and the orbit closure of a pattern Sturmian sequence is called a \emph{pattern Sturmian subshift}.
\end{definition}

\subsection{Three classes of known pattern Sturmian sequences}

Three classes of pattern Sturmian sequences are known, which we summarize here. (Hereafter, when we refer to the \emph{circle} or \emph{torus}, we mean the quotient space $\T = \R/\Z$ which can be canonically identified with the unit interval $[0, 1)$.)

\subsubsection{Simple circle rotation coding sequences}

\begin{definition}\label{def:coding}
A sequence $x$ is a \emph{circle rotation interval coding sequence} if there exist an irrational $\alpha \in [0,1)$, a partition of $[0,1)$ into $k$ intervals $I_0, \ldots, I_{k-1}$, and letters $a_0, \ldots, a_{k-1}$ not all equal so that
$x(n) = a_i$ if and only if $n\alpha \pmod 1 \in I_i$. In the case $k = 2$, $x$ is called a \emph{simple circle rotation coding sequence}. If we further assume that both intervals are half-open, and one of the intervals has length exactly $\alpha$, then $x$ is called a \emph{Sturmian sequence}.
\end{definition}

It was shown in \cite{Kamae-Zamboni-sequence_entropy} that simple circle rotation coding sequences with intervals of the form $[a,b)$ are pattern Sturmian, but in fact their proof can be easily adapted to all simple circle rotation coding sequences. We give a quick proof here for completeness. %\rp{This is the only proof in this section; is it OK to have one here? Does it share too much with the later proof of \Cref{lem:glue_up} or \Cref{prop:nonrecurrent_pattern_Sturmian}?}

\begin{lemma}\label{lem:simple_rotation_pattern_Sturmian}
    Every simple circle rotation coding sequence is pattern Sturmian.
\end{lemma}

\begin{proof}
Suppose that $x$ is a simple circle rotation coding sequence induced by $\alpha \notin \mathbb{Q}$ and partition 
$\xi = \{I_0, I_1\}$ of $[0,1)$ into intervals, and without loss of generality that $x(n) = i$ if and only if $n \alpha \pmod{1} \in I_i$. Denote the endpoints of the intervals $I_0, I_1$ by $y$ and $z$. Then for any window $\tau$ with $|\tau| = n$, if a word $w \in \{0,1\}^\tau$ is a $\tau$-subword of $x$, say $w = x(i+\tau)$, it means that the set 
$\bigcap_{j \in i + \tau} (I_{x(j)} - j\alpha)$ contains $0$, which implies that 
$\bigcap_{j \in \tau} (I_{x(i+j)} - j\alpha) \neq \varnothing$. Therefore, $|L_x(\tau)|$ is bounded from above by the number of nonempty sets in the partition $\bigvee_{j \in \tau} (\xi - j\alpha)$. 

If both $I_0$, $I_1$ are half-open, then all sets in this partition are themselves half-open intervals, determined by endpoints $\bigcup_{j \in \tau} \{y - j\alpha, z - j\alpha\}$. There are clearly at most $2n$ such points, and so $|L_x(\tau)| \leq 2n$. Since $\tau$ was arbitrary, $p^*_x(n) \leq 2n$, and so by \Cref{KZthm}, $x$ is pattern Sturmian.

If one of $I_0, I_1$ is closed, but $y - z \neq n\alpha$ for all $n \in \Z$, then all elements of $\bigvee_{j \in \tau} (\xi - j\alpha)$ are still intervals, and the same proof applies. The only remaining case is where
$y - z = n\alpha$ for some $n$; we can assume $n > 0$ by switching $y,z$ if necessary. Now, if $\tau$ contains elements with difference $n$, then we may get nonempty sets in $\bigvee_{j \in \tau} (\xi - j\alpha)$ which are not intervals but singletons; for instance, if $I_0 = [y, z]$, then we may have $I_0 \cap (I_0 - n\alpha) = \{z\}$. However, we can still show that $\left|\bigvee_{j \in \tau} (\xi - j\alpha)\right| \leq 2n$. 

Note that any singleton in $\bigvee_{j \in \tau} (\xi - j\alpha)$ corresponds to a pair $j, j' \in \tau$ for which 
$y - j\alpha = z - j'\alpha \Leftrightarrow j = j' + n$, and the number of such pairs is $\tau \cap (\tau - n)$. However, the number of intervals in this partition is bounded from above by $\bigcup_{j \in \tau} \{y - j\alpha, z - j\alpha\}$, whose cardinality is precisely $2n - |\tau \cap (\tau - n)|$. Therefore, the total number of sets in $\bigvee_{j \in \tau} (\xi - j\alpha)$ is still bounded from above by $2n$, completing the proof.
    
\end{proof}

It is important to note that when all intervals are half-open, it is easily shown that circle rotation interval coding sequences are recurrent. However, when at least one interval is closed, they can be nonrecurrent. For instance, if $\alpha < 1/2$, $I_0 = [0,\alpha]$, $I_1 = (\alpha, 1)$, $a_0 = 0$, and $a_1 = 1$, the induced simple circle rotation coding sequence begins with $00$ (since $0, \alpha \in I_0$), but contains no other consecutive $0$s (since $\alpha$ is irrational and $x, x+\alpha \in I_0$ happens only for $x = 0$). In fact, this phenomenon generalizes to the following characterization of nonrecurrent simple circle rotation coding sequences. 

\begin{proposition}
If $x$ is a simple circle rotation coding sequence with rotation number $\alpha$, then $x$ is nonrecurrent if and only if one of the intervals has the form $(k_1 \alpha, k_2 \alpha)$ or $[k_1 \alpha, k_2 \alpha] \bmod 1$ for some $k_1 \neq k_2 \in \N_0$.
\end{proposition}
\begin{proof}
For the ``only if'' direction, see \cref{prop:nonrecurrent_pattern_Sturmian}. To prove the ``if'' direction, let $[0, 1) = I_0 \cup I_1$ be the partition of the circle associated to $x$. If an interval has the form $[k_1 \alpha, k_2 \alpha]$, then the other interval has the form $(k_2 \alpha, k_1 \alpha)$ and vice versa. Therefore, without loss of generality, we can assume $I_0 = [k_1 \alpha, k_2 \alpha]$ and $k_1 < k_2$. Then $k_1 \alpha, k_2 \alpha \in I_0$, and if the length $|I_0|$ of $I_0$ is greater than $1/2$, we also get $k_1 + i(k_2 - k_1)\alpha \in I_0$ for $1 < i \leq 1/(1 - |I_0|)$.
% 
%Moreover, since the sequence $x$ would not change if we replaced $\alpha$ with $\alpha \bmod 1$, we assume $\alpha \in (0, 1)$.

Then, for the window $\tau = \{0, k_2 - k_1, 2(k_2 - k_1), \ldots, \lfloor \frac{1}{1 - |I_0|} \rfloor (k_2 - k_1)\}$, 
\[
    x(k_1 + \tau) = 0 \ldots 0.
\]
However, for all $n \neq k_1$, 
\[
    x(n + \tau) \neq 0 \ldots 0,
\]
and so $x$ is not recurrent.
\end{proof}

Regardless of the forms of intervals, all simple circle rotation coding sequences are pattern Sturmian, as shown in \cref{lem:simple_rotation_pattern_Sturmian}. This is in contrast to usual word complexity where being Sturmian requires both intervals to be half-open and have length $\alpha$ and $1 - \alpha$ respectively. %\Anote{I rewrote the previous two paragraph a bit. Please read.}

%Minimum word complexity, i.e. Sturmian, requires that both intervals be half-open and have length $\alpha$ and $1 - \alpha$ respectively, but for minimum maximal pattern complexity, these subtleties are not relevant.

%\Cnote{I changed this to $\N_0$ since we want to include the case where the interval starts at 0?}.

%\rp{So I haven't thought in detail, but do we indeed need a characterization? It seems that as of now we know: all 2-interval codings are pattern Sturmian, a recurrent pattern Sturmian with MEF circle must be a 2-interval coding, and a nonrecurrent pattern Sturmian containing an infinite minimal subsystem must be a (nonrecurrent) 2-interval coding. Isn't that enough to prove your new versions of Theorems A,B even without claiming a characterization? Or perhaps our existing results yield a characterization in 2-interval case...?}

\subsubsection{Nearly simple Toeplitz sequences}

For the second class of pattern Sturmian sequences, we need some definitions about Toeplitz sequences.

\begin{definition}
A sequence $x$ is a \emph{Toeplitz sequence with period structure $(n_k) \subseteq \N$} if 
\begin{enumerate}
    \item $n_k$ properly divides $n_{k+1}$ for all $k$,
    
    \item and there exists a partition of the form $\N_0 = \bigsqcup_{k,i} (a_{k,i} + n_k \N_0)$ where $x$ is constant on each infinite arithmetic progression $a_{k,i} + n_k \N_0$. 
\end{enumerate}
The number of \emph{holes at step $k$} in $x$ is $n_k \left(1 - \sum_{j = 1}^k \frac{|\{a_{j,i}\}_i|}{n_j}\right)$, i.e., the number of nonconstant arithmetic progressions in $x$ modulo $n_k$. If $x$ has $m$ holes at every step, it is called a \emph{$m$-hole Toeplitz sequence}. If $x$ is a $1$-hole Toeplitz sequence and $a_{k,i}$ is independent of $i$ for each $k$ (i.e., if for each $k$, all constant progressions in $x$ modulo $n_k$ take the same value), then $x$ is called a \emph{simple Toeplitz sequence}.
A ($m$-hole/simple) \emph{Toeplitz subshift with period structure $(n_k)$} is the closure of the orbit of a ($m$-hole/simple) Toeplitz sequence with that period structure.
\end{definition}

It is well-known that all Toeplitz sequences are uniformly recurrent, and so all Toeplitz subshifts are minimal.

%\Cnote{For this section below, starting sequences at 0 makes defining the example kind of ugly.  Is there a pretty way to write this example and start the sequence at 0 without adding -1 everywhere?} \rp{I think I have this done; see how it looks.} \Cnote{Should the first partition have $\N_0$ in the APs so that 0 is included in the partition?  And the second APs for x' just stay $\N$ since it does not partition the space?}

We note a subtlety in these definitions; not all elements of a Toeplitz subshift are themselves Toeplitz sequences. For instance, if $x$ with period structure $(2^k)$ is defined for $m \geq 0$ by $x(m) = i \pmod 2$ if and only if $m+1 = 2^i q$ for some odd integer $q$, then $x$ is a Toeplitz sequence associated to the partition $\N_0 = \bigsqcup (2^{k-1} - 1 + 2^k \N_0)$. However, the Toeplitz subshift $X = \overline{\OO(x)}$ contains the sequence $x'$ defined by $x'(0) = 1$ and $x'(m) = i \pmod 2$ if and only if $m = 2^i q$ for $m > 0$. Then $x'$ is just barely not a Toeplitz sequence itself, since its associated constant arithmetic progressions $2^{k-1} + 2^k \N$ do not partition $\N_0$. (Put another way, $x'(0)$ is not part of any constant arithmetic progression). In fact, a sequence in a Toeplitz subshift is not itself Toeplitz if and only if its associated arithmetic progressions do not completely cover $\N_0$.

%\Anote{Can we classify all of these outlier sequences?} \rp{So in the $k$-hole case, it's easy; it's just sequences where one or more of the $k$ holes does not approach infinity. In the irregular Toeplitz case, the `outliers' dominate in the sense that they have full measure. I'm hoping not to get into these details if it's not necessary.}

It was proved in \cite{Kamae-Zamboni-discrete} that simple Toeplitz sequences are pattern Sturmian. However, this proof was slightly incorrect. In \cite{Gjini_Kamae_Bo_Yu-Mei-toeplitz}, this proof is corrected and generalized to a slightly larger class. In fact, they give a complete characterization of pattern Sturmian $1$-hole Toeplitz sequences. (We do not give a definition of the technical condition $D((w?)^\infty, L) \leq 0$ here; see \cite[Page 1076]{Gjini_Kamae_Bo_Yu-Mei-toeplitz} for definitions.)

\begin{theorem}[{\cite[Theorem 2]{Gjini_Kamae_Bo_Yu-Mei-toeplitz}}]\label{onehole}
If $x$ is a $1$-hole Toeplitz sequence, it is pattern Sturmian if and only if it is either simple Toeplitz or it is a shift of the image of a simple Toeplitz sequence under a single constant-length morphism of the form $0 \mapsto w0, 1\mapsto w1$ satisfying the condition $D((w?)^\infty, L) \leq 0$ for $L \subseteq \{0, \ldots, |w|\}$.
\end{theorem}

%\rp{I actually don't think we need a different version of this; the coding sequence we get will be honestly Toeplitz; see line 6 of proof of \Cref{b1}.}

%\rp{Subtlety: this theorem is stated for 1-hole Toeplitz sequences, but we need it for subshifts. My current thought is to state the theorem for subshifts, and just comment that although it's stated for sequences in \cite{Gjini_Kamae_Bo_Yu-Mei-toeplitz}, since the subshifts are minimal it implies for the entire subshift.}
%\Anote{That's true and I agree with the solution.}

%\Anote{Add explanation why this constant length morphism implies the decomposition into residue classes where one residue class is a simple Toeplitz and other residues are constant used in Prop 5.4}

We call a sequence $x$ a \emph{nearly simple Toeplitz sequence} if it satisfies the conclusion of \Cref{onehole}, and the closure of the orbit of any such $x$ is a \emph{nearly simple Toeplitz subshift}. Nearly simple Toeplitz sequences are pattern Sturmian by \Cref{onehole}, so nearly simple Toeplitz subshifts (and all sequences in them) are as well. Finally, \Cref{onehole} also implies that a pattern Sturmian $1$-hole Toeplitz subshift $X$ is nearly simple, by just considering any $1$-hole pattern Sturmian sequence in $X$.

We note the following for future reference: if $x$ is a nearly simple Toeplitz sequence, then it is a shift of $\theta y$ for some simple Toeplitz $y$ and substitution $\theta: 0 \mapsto w0, 1 \mapsto w1$ of some length $C$. Therefore, if we define its subsequences
$x^{(i)}(n) := x(i + Cn)$ for $0 \leq i < C$, then all but one of the sequences $x^{(i)}$ is constant, and the nonconstant one is just $y$ itself, therefore simple Toeplitz.

\subsubsection{Almost constant sequences}

The final class of known pattern Sturmian sequences are nonrecurrent, and given in \cite{Kamae-Zamboni-sequence_entropy}. Specifically, they show that if $x$ is the characteristic function of a sequence $(s_k)$ with $s_{k+1} > 2s_k$ for all $k$, then $x$ is pattern Sturmian. The key fact about this sequence is that it is only nonzero on a very sparse set, which motivates the following definition.

For a set $S \subseteq \N_0$, the \emph{upper Banach density} of $S$ is 
\[
    d^*(S) = \lim_{N \to \infty} \max_{M \in \N_0} \frac{|S \cap [M, M + N)|}{N}.
\]

\begin{definition}
    A sequence $x$ is \emph{almost constant} if there exists an infinite set $S$ such that $d^*(S) = 0$ and $x$ is the characteristic function of $S$ or $\N_0 \setminus S$.
\end{definition}

While there are almost constant sequences which are pattern Sturmian, not all almost constant sequences are. For instance, if $S = \N_0 \setminus \{0, 1, 2, 3, 5, 9, 10\}$, then the characteristic function of $S$ is an almost constant sequence which begins with $00001011100$. This sequence contains all possible $\tau$-words for $\tau = \{0, 1, 2\}$, and so is not pattern Sturmian. 

Our Theorems \ref{mainthm:sturmian} and \ref{mainthm:sturmian-nonrecurrent} show that all pattern Sturmian sequences fall into one of these known categories: simple circle rotation coding sequences, sequences in nearly simple Toeplitz subshifts, or almost constant sequences. 

\section{Null subshifts, almost 1-1 extensions, and coding sequences}

Null topological dynamical systems were first considered in \cite{TNT} as those with zero topological sequence entropy for all sequences. In the setting of $\{0,1\}$-subshifts, this has a particularly simple interpretation. %(See \cite[Corollary 2.2]{Huang_Ye-combinatorial_lemmas} for the forward direction and \cite[Theorem 5.1]{GMtame} for the reverse.)

%\rp{Look for reference, see if it is iff}. \Anote{Here is the reference Casey found: \cite[Corollary 2.2]{Huang_Ye-combinatorial_lemmas}}

\begin{theorem}[{\cite[Corollary 2.2]{Huang_Ye-combinatorial_lemmas}}, 
{\cite[Theorem 5.1]{GMtame}}]
A $\{0,1\}$-subshift is null if and only if $\frac{\ln p^*_X(n)}{n} \rightarrow 0$. 
\end{theorem}
%\Anote{Do we have both if and only if directions? The ``if'' direction is proved in \cite[Corollary 1]{Kamae-Zamboni-sequence_entropy} and this is the direction we need.}

In particular, pattern Sturmian subshifts are null. %\Anote{We haven't defined pattern Sturmian subshifts.} 
In the measurable category, having zero sequence entropy for all sequences was shown by Kushnirenko (\cite{Kushnirenko}) to imply isomorphism to a group rotation. Infinite null subshifts cannot be topologically conjugate to rotations (see \Cref{lem:boundary_non_empty}), but they are extremely close. To say more about this,  we first need a definition.

%It is known that minimal null subshifts are very close to group rotations topologically (though not isomorphic, as for the corresponding concept using measurable entropy \Anote{this is a bit confusing}). \rp{rewrite sentence} 

\begin{definition}
Let $(X, T)$ and $(Y, S)$ be minimal topological dynamical systems. We say $(X, T)$ is an \emph{almost $1$-$1$ extension} of $(Y, S)$ if there exists a surjective continuous map $\phi: X \rightarrow Y$ such that $\phi \circ T = S \circ \phi$ and $\phi$ is injective somewhere, i.e. there exists $y \in Y$ with $|\phi^{-1}(\{y\})| = 1$.
\end{definition}

In fact, for almost $1$-$1$ extensions (between minimal systems), the map $\phi$ is generically injective, i.e. the set of $y \in Y$ with singleton preimages is a dense $G_\delta$ set.

Every minimal topological dynamical system $(X,T)$ has a maximal group rotation $(G, +g_0)$ occurring as a factor, which is called the \emph{maximal equicontinuous factor or MEF}. Since $X$ is compact and $(X,T)$ is minimal, $G$ is compact and the MEF rotation $(G, +g_0)$ is \emph{monothetic}, meaning that $\{ng_0\}_{n \in \N_0}$ is dense in $G$.

% of a minimal topological system is the largest factor that is isomorphic to a compact monothetic group rotation (or group rotation for short).}

We do not need a general treatment of the theory of maximal equicontinuous factors here (see, for example, \cite{Auslander88, MEFref}), but note that the MEF is uniquely determined up to group isomorphism.

\begin{theorem}[{\cite[Theorem 4.3]{Huang_Li_Shao_Ye-null_systems}}, {\cite[Corollary 7.16]{Kerr_Li-independence_Cstar}}]
\label{kerrli}
If $X$ is minimal and null, then $X$ is an almost $1$-$1$ extension of its MEF.%, which is a compact monothetic group rotation $(G, g)$.%, which is its maximal equicontinuous factor. \Anote{``...1-1 extension of its MEF.''}
\end{theorem}

In fact, in the case where $X$ is a $\{0,1\}$-subshift, such an almost $1$-$1$ extension can be represented explicitly in terms of symbolic coding of the associated rotation, similarly to \Cref{def:coding}.

The following is folklore, and seems to be implictly present in work of Downarowicz and others on semi-cocycles (for example see \cite{cocycles}), but we include it for ease of reference in future results.

\begin{theorem}[see {\cite[Theorem 6.4]{cocycles}}]
\label{coding}
Suppose that $\phi: X \rightarrow Y$ is an almost $1$-$1$ extension, $(X, \sigma)$ is a minimal $\{0, 1\}$-subshift, and $(Y, T)$ is a topological dynamical system. Then there exists a partition
$Y = U_0 \cup U_1 \cup B$ where $U_0, U_1$ are open and nonempty, $B$ is the common boundary of $U_0, U_1$, and there exist $x_0 \in X$ and $y_0 \in Y$ so that $x_0$ comes from coding the orbit of $y_0$ as follows: for all $n \geq 0$, $T^n y_0 \notin B$, and for $i \in \{0, 1\}$, $x_0(n) = i$ if and only if $T^n y_0 \in U_i$.
\end{theorem}

\begin{proof}
This is relatively straightforward: simply define $K_i = \phi([i])$ for $i \in \{0,1\}$. Then each $K_i$ is compact and nonempty, their union is $Y$, and their intersection is first category by definition of almost $1$-$1$. Since it is also closed, it is nowhere dense; represent this intersection by $I$. Then define $V_0 = K_1^c$ and $V_1 = K_0^c$, and $\{V_0, V_1, I\}$ forms a partition of $Y$. 

Now, define $I_i = I \cap \partial V_i$ for $i \in \{0,1\}$. Since $I$ was nowhere dense, $I = I_0 \cup I_1$. Define $B = I_0 \cap I_1$ and
$U_i = V_i \cup (I_i \setminus B)$. Then each $U_i$ is open; if $y \in V_i$, then $y$ has a neighborhood in $V_i \subseteq U_i$ since $y$ is open, and if $y \in I_i \setminus B$, then by definition there is a neighborhood $W$ of $y$ containing no point from $V_{1-i}$, meaning that every point of $W$ is neither in $V_{1-i}$ or $I_{1-i}$, and so is in $U_i$. $B$ is still closed and nowhere dense, and every neighborhood of a point in $B$ contains points of both $U_i \subseteq V_i$ by definition, so $B$ is the common boundary of $U_0, U_1$.

Since $I$ was nowhere dense and $T$ is continuous, the union $\bigcup_n T^n I$ is first category, and so there exists $y_0$ in its complement. Then for each $n \geq 0$, $T^n y_0 \in V_0 \cup V_1$, and if we denote by $x_0$ any point in $\phi^{-1}(y_0)$, then by definition
$x_0(n) = i$ if and only if $T^n y_0 \in V_i \subseteq U_i$, completing the proof.
\end{proof}

In the case where $(Y, T)$ is a group rotation, without loss of generality, we can assume that $y_0$ is the identity (by translating the partition by $y_0^{-1}$), and in this case we make the following definition.

\begin{definition}
For a minimal null $\{0,1\}$-subshift $X$ with the maximal equicontinuous factor $Y$, we refer to any partition satisfying \Cref{coding} for $y_0$ equal to the identity as an \emph{MEF partition} associated to $X$.

%\Anote{If we only deal with $\{0, 1\}$-subshift in this paper, maybe say so in the intro and later just simply say ``For a minimal null subshift...''} \rp{Yes, we should decide this. I guess the issue is that our linear results seem to hold on arbitrary alphabets without much change. How much is it worth making the notation simpler to have a technically weaker result? We could just do $\{0,1\}$ and then add a remark after the main sublinear minimal theorem saying it holds for more general alphabets by exactly same proof.}
\end{definition}

We note for future reference that when the MEF is infinite, in fact one can interpret all points of $X$ as coding sequences (as opposed to just the distinguished $x_0$ from \Cref{coding}).

\begin{theorem}\label{gencoding}
    If $X$ is an infinite minimal subshift with MEF $(G, +g_0)$ and MEF partition $\{U_0, U_1, B\}$, 
    then there exists a partition $\{B_0, B_1\}$ of $B$ and $y \in Y$ so that $x$ comes from coding the orbit of $y$ in the following sense:
\[
    \text{ for } i \in \{0, 1\}, \, \, x(n) = i \Longleftrightarrow y + ng_0 \in U_i \cup B_i.
\]
\end{theorem}

\begin{proof}
    By definition of an MEF partition, there exists $x_0 \in X$ which comes from coding the orbit of the identity in the following sense: 
    \[
    x_0(n) = i \Longleftrightarrow ng_0 \in U_i.
    \]
We recall that in particular, no $ng_0$ lies in $B$. Now, consider any $x \in X$. Since $X$ is minimal, there exists a sequence $(n_k)$ so that $\sigma^{n_k} x_0 \rightarrow x$. By compactness and passing to a subsequence, we can assume that $n_{k} g_0$ converges to some limit $y \in Y$. 

Now, by definition of the shift,
\[
(\sigma^{n_k} x_0)(n) = i \Longleftrightarrow (n + n_k) g_0 \in U_i.
\]
Since $\sigma^{n_k} x_0 \rightarrow x$, this means that $x(n) = i$ if and only if $(n + n_k) g_0 \in U_i$ for sufficiently large $k$. Note that $(n + n_k) g_0 \rightarrow n\alpha + y$. If $ng_0 + y$ is in either open set $U_i$, then clearly $x(n) = i$. The only remaining case is when $ng_0 + y \in B$. Since $X$ is infinite, $Y$ is infinite. We know that $g_0$ is a monothetic rotation, and so for each element $b \in B$, there exists at most one $n$ for which $n g_0 + y = b$; if such $n$ exists, then assign 
$b$ to $B_{x(n)}$. Then by definition, for every $n$, $x(n) = i$ if and only if $ng_0 \in (U_i \cup B_i) - y$. For any $b$ which are not equal to any $ng_0 + y$, assign to either $B_0$ or $B_1$ arbitrarily; this completes the proof.

\end{proof}

It is fairly straightforward how recurrent simple circle rotation coding sequences can be viewed in this context; the MEF partition has open sets given by interiors of the coding intervals and the boundary is the set of endpoints. It will be useful for future reference to explicitly define the MEF partition for Toeplitz sequences.

\begin{lemma}\label{toeMEF}
    If $x$ is a Toeplitz sequence with period structure $(n_k)$, then the subshift $X = \overline{\OO(x)}$ is minimal and an almost $1$-$1$ extension of its maximal equicontinuous factor, which is addition by $1$ on the odometer
    \[
    \mathcal{O} = \lim_{\longleftarrow} \faktor{\Z}{n_k \Z}.
    \]
%\Anote{compare above with $\varprojlim \Z/n_k \Z$, which one do we prefer?}
In addition, if $x$ is defined by partition $\N_0 = \bigsqcup_{k,i} (a_{k,i} + n_k \N_0)$ where $x$ is constant on every $a_{k,i} + n_k \N_0$, then $X$ can be associated with the MEF partition where the clopen subset of elements $y \in \mathcal{O}$ with $y(k) = a_{k,i}$ is a subset of $U_j$ if and only if $x = j$ on $a_{k,i} + n_k \N_0$. Then $B$ is the set of all $y \in \mathcal{O}$ where for all $k$, $x$ is not constant on the arithmetic progression $y(k) + n_k \N_0$.
\end{lemma}

\begin{proof}
This essentially follows from the definitions. First, define $U_0, U_1, B$ as in the lemma; they are well-defined since the progressions $a_{k,i} + n_k\N_0$ are disjoint, meaning that no $y \in \mathcal{O}$ can have $y(k) = a_{k,i}$ and $y(k') = a_{k',i'}$ for distinct pairs $(k,i) \neq (k', i')$. 

Every $n \in \N_0$ is part of exactly one of the arithmetic progressions $a_{k,i} + n_k \N_0$; say that $x$ has constant value $j$ on this progression. Then $x(n) = j$, and $n \cdot 1 = n \in \mathcal{O}$ has $k$th coordinate $n \pmod{n_k} = a_{k,i}$, so by definition, $n \in U_j$ as an element of $\mathcal{O}$. Therefore, $x_0 = x$ arises from coding the orbit of $y_0 = 0$ as in \Cref{coding}. 

It remains only to check that the claimed $B$ is both the common boundary of $U_0$ and $U_1$ and the complement of their union; this is straightforward and left to the reader.

\end{proof}

We present the following example to illustrate \cref{toeMEF}. 
\begin{example}
   Suppose that $x$ is a sequence with alphabet $\{0,1\}$ defined as follows: for any $n \in \mathbb{N}_0$, if we write $n+1 = 3^r s$ for $s$ not a multiple of $3$, then 
    $x(n) = (s \pmod 3) -1$. Then
    \[
    x = .010011010010011010010011011\ldots
    \]
Then we can see that $x$ is a nonsimple 1-hole Toeplitz with $n_k = 3^k$, $a_{k,1} = 3^{k-1} - 1$, $a_{k,2} = 2 \cdot 3^{k-1} -  1$, and associated partition 
$\mathbb{N}_0 = \bigcup_{k \in \N} (3^{k-1} - 1 + 3^k \N_0) \cup (2 \cdot 3^{k-1} - 1 + 3^k \N_0)$. The sequence $x$ is constant of all $0$s on each $3^{k-1} - 1 + 3^k \N_0$ and all $1$s on each 
$2 \cdot 3^{k-1} - 1 + 3^k \N_0$. Therefore, by \Cref{toeMEF}, the orbit closure $X$ of $x$ has MEF the odometer $\displaystyle \mathcal{O} = \lim_{\longleftarrow} \faktor{\Z}{3^k \Z}$. The associated MEF partition is $\{U_0, U_1, B\}$ where $U_0$ consists of all $y \in \mathcal{O}$ beginning with a string of $-1$ followed by $y(k) = 3^{k-1} - 1$, $U_1$ consists of all $y \in \mathcal{O}$ beginning with a string of $-1$ followed by $y(k) = 2 \cdot 3^{k-1} - 1$, and $B = \{(-1, -1, -1, \ldots)\}$.
\end{example}

We finish with a simple corollary of \Cref{toeMEF}.

\begin{corollary}\label{mhole}
If $X$ is a Toeplitz subshift, then it is an $m$-hole Toeplitz subshift if and only if there exists an MEF partition for $X$ which has $|B| = m$.
\end{corollary}

\begin{proof}
    If $X = \overline{\OO(x)}$ for some $m$-hole Toeplitz sequence $x$, then for each $k$ there are $m$ residue classes $\bmod$  $n_k$ of $x$ which are not equal to some $a_{k,i}$, and taking limits yields exactly $k$ elements of $B$ as described in \Cref{toeMEF}. 

    Conversely, suppose that a Toeplitz subshift $X$ has MEF an odometer 
    $\mathcal{O} = \varprojlim \faktor{\Z}{n_k\Z}$ and an MEF partition with $|B| = m$. For each $k$, define the set $B_k$ of $0 \leq i < n_k$ so that $U_0$, $U_1$ each contain elements with $k$th coordinate $i$. Then, by definition, $B$ is the set of limits of elements of $B_k$ as $k \rightarrow \infty$, and so for sufficiently large $k$, $|B_k| = m$. Then, coding the orbit of $0$ by the partition $\{U_0, U_1, B\}$ yields $x_0 \in X$ which is $m$-hole Toeplitz by definition; nonconstant residue classes of $x$ modulo $n_k$ correspond to elements of $B_k$, so for large enough $k$ there are exactly $m$ of them. Since $X$ is minimal, $X = \overline{\OO(x_0)}$, completing the proof.
\end{proof}

\section{Non-superlinear maximal pattern complexity: minimal case}

\subsection{Connection between maximal pattern complexity and MEF}

Our main observation is the size of the boundary for an associated MEF partition directly influences the maximal pattern complexity.

%We now claim that the size of $B$ influences the maximal pattern complexity of minimal null $X$ in a natural way.

\begin{theorem}\label{main}
If $X$ is a minimal null subshift with MEF partition $\{U_0, U_1, B\}$ with $|B| \geq k$, then 
\[
    \inf_{n \in \N} (p^*_X(n) - kn) > -\infty.
\]
%is bounded from below. %which is an almost $1$-$1$ extension of the group rotation $(G, g)$ and $(U_0, U_1, B)$ is a partition associated to this extension as in \Cref{coding}, and $|B| \geq k$, then $p^*_n(X) - kn$ is bounded from below.
\end{theorem}

\begin{proof}
Suppose that $\phi: X \rightarrow G$ is the almost $1$-$1$ extension of the MEF $(G, +g_0)$ guaranteed by \Cref{kerrli} and $U_0, U_1, B, x_0$ are as in 
\Cref{coding}; recall that $y_0$ is the identity by definition of MEF partition. For some $k \in \N$, define distinct $b_1, \ldots, b_k \in B$ and $\epsilon > 0$ so that the $\epsilon$-balls around $b_i$ are pairwise disjoint. 

For any finite $\tau \subseteq \N_0$, define the partition $\xi_{\tau}$ of $G$ into the (possibly empty) open subsets $U_{w,{\tau}} = \bigcap_{s \in {\tau}} (U_{w(s)} - sg_0)$ for each $w \in \{0,1\}^{\tau}$ and the complement $C_{\tau} = \bigcup_{s \in {\tau}} (B - sg_0)$. The significance is that
$ig_0 \in U_{w,\tau} \Longrightarrow x_0(i + \tau) = w$. Since $g_0$ is a monothetic rotation and $U_{w,\tau}$ is open, this means that $U_{w,{\tau}} \neq \varnothing \Longrightarrow w \in L_X({\tau})$. 

Since $\phi$ is a function, the partitions $\xi_{\tau}$ separate points in $y$, and so there exists ${\tau}$ so that
every nonempty $U_{w,{\tau}}$ for ${\tau}$ has diameter less than $\epsilon$. We now define sets ${\tau}_n$, with $|{\tau}_n| = |{\tau}| + n$, ${\tau}_0 = {\tau}$, and 
${\tau}_0 \subseteq {\tau}_1 \subseteq {\tau}_2 \cdots$, so that $|L_X({\tau}_{n+1})| \geq |L_X({\tau}_n)| + k$ for all $n \geq 0$, which completes the proof.

Suppose that ${\tau}_n \supseteq {\tau}_0$ has been defined, so that in particular the nonempty sets $U_{w,{\tau}_n}$ have diameter less than 
$\epsilon$. Each set $C_{{\tau}_n} - b_i$ is nowhere dense, so the union $\bigcup_{i=1}^k (C_{{\tau}_n} - b_i)$ is nowhere dense. Since the set $\{-mg_0 \ : m \in \N\}$ is dense, there must exist infinitely many $m$ so that $-mg_0$ is not in $\bigcup_{i=1}^k (C_{{\tau}_n} - b_i)$, which implies that $C_{{\tau}_n} + mg_0$ contains no $b_i$. Choose such an $m$ greater than $\max {\tau}_n$, and define ${\tau}_{n+1} = {\tau}_n \cup \{m\}$. 

Since $C_{{\tau}_n} + mg_0$ contains no $b_i$, each $b_i - mg_0$ is inside some $U_{w_i,{\tau}_n}$, and since each has diameter less than $\epsilon$, the $w_i$ are distinct. Then, by definition of boundary points, each $U_{w_i, {\tau}_n}$ contains points of both $U_0 - mg_0$ and $U_1 - mg_0$, meaning that the intersections $U_{w_i a,{\tau}_{n+1}} = \left(\bigcap_{s \in {\tau}_n} (U_{w_i(s)} - sg_0)\right) \cap (U_a - mg_0)$ are both nonempty, and so the concatenations $w_i 0, w_i 1$ are both in $L_X({\tau}_{n+1})$. Since every word in $L_X({\tau}_n)$ has at least one extension to ${\tau}_{n+1}$ and $k$ words have two extensions, $|L_X({\tau}_{n+1})| \geq |L_X({\tau}_n)| + k$, completing the proof.

\end{proof}

We now have the following immediate corollary.

\begin{corollary}\label{grpcor}
If $X$ is a minimal subshift with non-superlinear maximal pattern complexity, then the boundary set $B$ is finite for any associated MEF partition.
\end{corollary}

\subsection{Possible MEF for a minimal subshift of non-superlinear maximal pattern complexity}

The goal of this subsection is to prove \cref{thm:possible_MEF_nonsuperlinear} which states that an  infinite minimal subshift with non-superlinear maximal pattern complexity must have MEF either the product of a circle with a finite cyclic group or an odometer. %\rp{Do we want to assume aperiodicity in this section to remove the finite cyclic group case?}

We start with the following nice consequence of the Peter-Weyl theorem \cite{Peter_Weyl-compact_group}:
\begin{lemma}[\cite{Peter_Weyl-compact_group}]\label{thm:peter-weyl}
Every compact Hausdorff topological group is an inverse limit of Lie groups.
\end{lemma}

Let $(X, \sigma)$ be a minimal subshift of non-superlinear maximal pattern complexity and let $(G, +g_0)$ be its maximal equicontinuous factor. 
By \cref{thm:peter-weyl}, $G$ is an inverse limit of Lie groups, say $G = \varprojlim G_i$. Since $G$ is metrizable, we can take $(G_i)$ to be a sequence. Let $\pi_i: G \to G_i$ and $\pi_{i}^{j}: G_j \to G_i$ be the associated bonding maps. We can always assume that $\pi_i, \pi_{i}^j$ are surjective. Compact monothetic (hence abelian) Lie groups have the form $\T^d \times \faktor{\Z}{k\Z}$ where $d \in \N_0$ and $k \in \N$. Thus $G_i = \T^{d_i} \times \faktor{\Z}{k_i\Z}$ for all $i$. Since $\pi_{i}^j$ is surjective, $(d_i)$ is a nondecreasing sequence.

First, we will treat the special case that $G_i$ is connected (i.e. $k_i = 1$) for all $i$. In this case, if $d_i = 0$ for all $i$, then $G$ is just a single point. Now assume $d_i \geq 1$ for all $i$.
Each bonding map $\pi_i^{i+1}: \T^{d_{i+1}} \to \T^{d_i}$ is a $t_i$-to-$1$ covering map where $t_i \in \N \cup \{\infty\}$. If all but finitely many $t_i$ is equal to $1$, then $G$ is a torus. On the other hand, if $t_i \geq 2$ for infinitely many $i$, then $G$ is called a \emph{solenoidal space}.

\begin{lemma}\label{lem:solenoid_or_dim2_cannot_disconnect}
If $G$ is a torus of dimension $\geq 2$ or a solenoidal space, then $G$ cannot be partitioned into two nonempty open sets and a countable set.
\end{lemma}
\begin{proof}
For contradiction, assume $G = U_0 \cup U_1 \cup B$ is a partition where $U_0, U_1$ are nonempty open sets, and $B$ is countable.

Case 1: $G$ is a torus of dimension $\geq 2$. Let $x \in U_0$ and $y \in U_1$. There are uncountably many disjoint paths in $G$ connecting $x$ and $y$. As a result, there is such a path $P$ that does not intersect $B$. Now $P = (P \cap U_0) \cup (P \cap U_1)$ is a partition of $P$ into two disjoint open subsets (relatively to $P$) and this contradicts the fact that $P$ is connected.

Case 2: $G$ is a solenoidal space. By \cite[Corollary 5.11]{McCord-inverseLimitCovering}, $G$ has uncountably many (disjoint) path components and so there is a path component $P$ that does not intersect $B$. By \cite[Theorem 5.8]{McCord-inverseLimitCovering}, each path component, and so $P$, is dense in $G$. Therefore, $P \cap U_0$ and $P \cap U_1$ are nonempty. As in the previous case, this leads to a contradiction.
\end{proof}

The following lemma says that the boundary $B$ in an associated MEF partition of a nonperiodic, minimal, non-superlinear maximal pattern subshift is nonempty.

\begin{lemma}\label{lem:boundary_non_empty}
If $X$ is an infinite minimal subshift with non-superlinear maximal pattern complexity and $G = U_0 \cup U_1 \cup B$ is the associated partition of the MEF, then $|B| > 0$.
\end{lemma}
\begin{proof}
%\Anote{Refine the following proof...}
Let $\phi: X \to G$ be the factor map from $(X, \sigma)$ to its maximal equicontinuous factor $(G, +g_0)$. If $|B| = 0$, then $\phi$ is an isomorphism. Note that the rotation on $G$ is an isometry (i.e. $d_G(x + n g_0, y + n g_0) = d_G(x, y)$ for $x, y \in G$ and $n \in \N$). %\rp{with respect to which metric? I think there's a canonical one but probably need to say something.}
On the other hand, since $X$ is a subshift, it is expansive (i.e. there exists $C > 0$ such that for $x \neq y \in X$, $d_X(\sigma^n x, \sigma^n y) > C$ for some $n \in \N$) unless $X$ is a finite rotation. However, $X$ is infinite and so this is not possible.

%If $B  = \varnothing$ then the factor map is a conjugacy. Expansiveness preserved under conjugacy, so your rotation is an isometry and expansive, which implies that the group is discrete and therefore finite by compactness.

%Since $X$ is a subshift and $G$ is a group rotation, they are isomorphic if and only if they are finite rotations. Thus, $|B| > 0$.    
\end{proof}

We can now complete our characterization of possible maximal equicontinuous factors for minimal subshifts with non-superlinear maximal pattern complexity.

% \begin{theorem}
% The maximal equicontinuous factor of a minimal pattern Sturmian shift is either an odometer or an irrational rotation on $1$-dimensional torus.
% \end{theorem}

\begin{proposition}\label{thm:possible_MEF_nonsuperlinear}
If $G$ is the MEF of an infinite minimal subshift with non-superlinear maximal pattern complexity, then $G$ is either a product of a circle with a finite cyclic group or an odometer.
\end{proposition}
\begin{proof}
Let $(X, \sigma)$ be an infinite minimal subshift with non-superlinear maximal pattern complexity and $(G, +g_0)$ be its maximal equicontinuous factor. 
As discussed after \cref{thm:peter-weyl}, $G = \varprojlim G_i$ where $G_i = \T^{d_i} \times \faktor{\Z}{k_i\Z}$ for all $i$.  Let $\pi_i: G \to G_i$ and $\pi_{i}^j: G_j \to G_i$ denote the associated (surjective) bonding maps.

% Furthermore, $G$ contains (possibly uncountably many) connected components. Each of the connected component is an inverse limit of the form $\varprojlim \T^{d_i}$ where the bonding maps $\pi_i^j: \T^{d_j} \to \T^{d_i}$ have the form
% \[
%     \pi_i^j(x_1, \ldots, x_{d_j}) = (a_{11} x_1 + \cdots+ a_{1, d_j} x_{d_j}, \ldots, a_{d_i, 1} x_1 + \cdots + a_{d_i, d_j} x_{d_j})
% \]
% where $a_{mn} \in \Z$.

If $d_i = 0$ for all for all $i$, then $G$ is either a finite cyclic group or an infinite odometer $\varprojlim \faktor{\Z}{k_i\Z}$. The finite case is impossible since $X$ is infinite.

Next we claim that it cannot be true that that $d_i \geq 2$ for large $i$. Let $G = U_0 \cup U_1 \cup B$ be the associated partition of the MEF $G$. By \Cref{main} and \cref{lem:boundary_non_empty}, $B$ is a nonempty, finite set. Let $H$ be a connected component of $G$. Then $H$ is an inverse limit of the form $\varprojlim \T^{d_i}$ where the bonding maps $\pi_i^j: \T^{d_j} \to 
\T^{d_i}$ are surjective continuous homomorphisms. By \cref{lem:solenoid_or_dim2_cannot_disconnect}, $H$ cannot be partitioned into two disjoint open sets and a finite set.
% have the form
% \[
%     \pi_i^j(x_1, \ldots, x_{d_j}) = (a_{11} x_1 + \cdots+ a_{1, d_j} x_{d_j}, \ldots, a_{d_i, 1} x_1 + \cdots + a_{d_i, d_j} x_{d_j})
% \]
% where $a_{mn} \in \Z$. 
% Depending on the values of the $a_{mn}$, $H$ is either a torus of dimension $\geq 2$ or contains a copy of a solenoid. (More precisely, if all $a_{mn} \in \{0, \pm 1\}$ for all but finitely many bonding maps $\pi_{i+1, i}$, then $H$ is a torus, otherwise, $H$ contains a copy of a solenoid.) In any case, $H$ cannot be written as the union of two disjoint open sets and a finite set. (The torus case is obvious; the solenoid case was proved in \cref{lem:solenoid_cannot_be_disconntected}.) \rp{This is a great proof! But I don't understand why containing a solenoid precludes being disconnectable by removing finite set. Is the truth stronger, that $H$ has a solenoid as a connected component or something?} 
Thus if $U_0$ intersects $H$, $U_1$ must be disjoint from $H$ and vice versa. This holds for every connected component $H$ of $G$.

Without loss of generality, assume $U_0$ and $B$ intersect the connected component $H$. As discussed above, $U_1$ and $H$ are disjoint. However, as $B$ is the common boundary of $U_0$ and $U_1$, and $B$ is finite, there is a sequence of connected components $(H_i)$ of $G$ such that
\begin{enumerate}
    \item $H_i \subseteq U_1$
    \item \label{item:BH_closure} $B \cap H$ is in the closure of $\bigcup_i H_i$.
\end{enumerate}
Let $H_0$ be the connected component of $G$ containing the identity $0_G$. Then we can write $H = g + H_0$ and $H_i = g_i + H_0$ for some $g, g_i \in G$. Let $b = g + h_0 \in B \cap H$ where $h_0 \in H_0$. Item \eqref{item:BH_closure} implies that there exists a sequence $h_i \in H_0$ such that $g_i + h_i \to g + h_0$. Thus, for any $h \in H_0$, $g_i + h_i + h \to g + h_0 + h$. Since $h \in H_0$ is arbitrary, the entire $H = g + H_0$ is in the closure of $\bigcup_i H_i$.
Therefore, the common boundary of $U_0$ and $U_1$ contains $H$ instead of just $B$, a contradiction. 

It remains to deal with the case $d_i = 1$ for all but finitely many $i$. Similarly to the case $d_i \geq 2$, each connected component of $G$ is an inverse limit of $1$-dimensional tori and so is a $1$-dimensional torus or a solenoidal space. If they are solenoidal spaces, then arguing as before, in light of \cref{lem:solenoid_or_dim2_cannot_disconnect}, the common boundary of $U_0$ and $U_1$ cannot be finite. Thus each connected component is the torus $\T$. %Again by a similar argument as in the case $d_i \geq 2$, there must be a connected component $H$ such that $U_0, U_1$ and $B$ all intersect $H$. 
Consider two cases:
\begin{enumerate}
    \item Case 1: $G$ has finitely many copies of the torus $\T$, i.e. $G \cong \T \times \faktor{\Z}{k\Z}$ for some $k \in \N$. We are done. 
    
    %If $k \geq 2$, then by considering the window $\tau = \{0, k, 2k, \ldots, (n-1)k\}$, we can see that $(X, T)$ is not pattern-Sturmian [[Need to elaborate here...]] Thus $k = 1$ and so $(Y, T)$ is an irrational on $1$-dimensional torus.

    \item Case $2$: $G$ has infinitely many copies of the torus $\T$, i.e. $G \cong \T \times \mathcal{O}$ with the product topology. Then two boundary points must be in one torus, say $\T \times \{y_0\}$, and each other torus belongs to either $U_0$ or $U_1$. Since the addition on the odometer $\mathcal{O}$ is not periodic, the orbit of every point in $\T \times \{y_0\}$ visits $\T \times \{y_0\}$ only at the time $n = 0$ and never returns. As a result, the (infinite) collection of points in $\T \times \{y_0\}$ corresponds to only two possible coding sequences in $X$ and this contradicts the fact that $X \to G$ is a (surjective) factor map.

 %   It follows that all points in any torus other than $T_0$ have the same coding \rp{Is this true? It would seem that the orbit of $T_0$ could intersect countably many tori, and so points in any of those tori could have multiple coding sequences. But that still is perhaps good enough?} and so $X$ is not an almost 1-1 extension of $G$, a contradiction.
  %  \Anote{I fixed it.}
\end{enumerate}
\end{proof}

% In light of the previous lemma, it remains to deal with the situation that the boundary set $B$ has cardinality $1$ or $2$.

% \begin{lemma}
% If $G$ is a metrizable group and $G_0$ is the identity component of $G$, then the covering dimension of $G$ is less than or equal to the covering dimension of $G_0$.
% \end{lemma}
% \begin{proof}
% $G_0$ is a closed normal subgroup of $G$ and the quotient $H = G/G_0$ is a totally disconnected group. As topological space, $G$ is the Cartesian product $G_0 \times H$. (As a group, $G = G_0 \rtimes H$, a semidirect product between $G_0$ and $H$.) 
% Since $G$ is metrizable, it is well-known that 
% \[
%     \dim (G) = \dim(G_0 \times H) \leq \dim(G_0) + \dim(H).
% \]
% Since $H$ is totally disconnected, $\dim(H) = 0$ and so our lemma follows.
% \end{proof}

% \begin{conjecture}\label{conj:dim_MEF_pattern_Sturmian}
% The covering dimension of the MEF of a pattern-Sturmian shift is less than $2$. 
% \end{conjecture}

% \begin{remark}
% Locally compact Hausdorff space is zero dimensional if and only if it is totally disconnected. A totally disconnected compact group is profinite. 
% Therefore, if the MEF of a pattern-Sturmian shift is zero dimensional, it must be an odometer. This together with \cref{conj:dim_MEF_pattern_Sturmian} reduces our problem to the case of dimension $1$.
% \end{remark}

\subsection{Characterizing recurrent sequences with non-superlinear maximal pattern complexity}

%We can therefore use the MEF to characterize \rp{I think this is my wrong language; I was hoping to characterize, but I think we won't get a true if and only if condition} infinite minimal subshifts of non-superlinear maximal pattern complexity.

For the rest of this subsection, we assume that $x$ is a nonperiodic, recurrent sequence with non-superlinear maximal pattern complexity. To use the results of the previous subsection, we need to know that $x$ generates a minimal subshift, i.e., is uniformly recurrent. This is implied by \Cref{mainthm:recnotmin}, which we prove now.% \rp{This proof is written for two-sided shifts; I checked previously that it still holds for one-sided, but I'll need to rewrite some things later.}

\begin{proof}[Proof of \cref{mainthm:recnotmin}]
%\begin{proposition}\label{rec}
%If $x \in \{0,1\}^\mathbb{Z}$ is not eventually periodic, recurrent, and not uniformly recurrent, then
%\[
%\liminf \frac{p^*_X(n)}{n \ln n} > 0.
%\]
%\end{proposition}

%\rp{warning: I think this proof may rely on two-sided and/or our definition of recurrent. If we're 2-sided, then does recurrent mean positively recurrent or usual recurrent? I don't think it's going to be an issue but we need to be careful for this and \Cref{lem:nonrecurrent_infinite_minimal}.}
Suppose that $x$ is recurrent but not uniformly recurrent. Then $X = \overline{\mathcal{O}(x)}$ properly contains a minimal subshift $Y$, and so there exists some $N$ for which $L_X(N) \setminus L_Y(N) \neq \varnothing$.

Define a function $\phi: X \rightarrow \{0,1\}^{\N_0}$ as follows: for $i \in \N_0$, $(\phi(x))(i) = 1$ if and only if
$x([i, i+N)) \in L_Y(N)$. Then $\phi$ is a so-called sliding block code, meaning that the image $x' := \phi x$ of the recurrent sequence $x$ is also recurrent.
We note that since $Y \subseteq X$, $x'$ contains arbitrarily long strings of $1$s, and since $L_X(N) \setminus L_Y(N) \neq \varnothing$, $x'$ contains a $0$.  
This implies that the block $01^n$ is a subword of $x'$ for all $n$, and so must appear twice in $x'$ by recurrence, implying that $1^n 0$ is also a subword of $x'$ for all $n$. 

For the rest of this argument we bound maximal pattern complexity of $x'$.
We will construct a sequence of windows $\tau_n \subseteq \mathbb{N}$ with $|\tau_n| = 2^n$, beginning with 
$\tau_0 = \{0\}$.

For any $n$, suppose that $\tau_n$ has been defined, and that $M$ is sufficiently large so that
$x'([0, M])$ contains all $\tau_n$-words in $L_{x'}(\tau_n)$. By recurrence, we may choose 
$K > \max(\tau_n)$ so that $x'([K, K+M]) = x'([0, M])$. Define $\tau_{n+1} = \tau_n \cup (K + \tau_n)$.

By definition of $K$, for every word $w \in L_{x'}(\tau_n)$, we have $ww \in L_{x'}(\tau_{n+1})$. 
For any $w$ containing a $0$, if $ww$ were the only word in $L_{x'}(\tau_{n+1})$ beginning with $w$, then every copy of $w$ in $x'$ would force a $w$ exactly $K$ units later, and this would continue indefinitely. This is impossible since it contradicts $x'$ containing arbitrarily long strings of $1$s. 

So, for every $w \in L_{x'}(\tau_n)$ except $1^{\tau_n}$, $w$ is a prefix of at least two words in 
$L_{x'}(\tau_{n+1})$. We now choose $D$ greater than the diameter of $\tau_{n+1}$. Since $1^D 0$ is a subword of $x'$, by beginning at the left edge of this word and moving to the right, for each $1 \leq i \leq |\tau_n| = 2^n$, we can find a word in $L_{x'}(\tau_{n+1})$ beginning with $1^{\tau_n}$ and with leftmost $0$ at the $i$th location within the second copy of $\tau_n$. % \rp{It's not a huge deal, but this doesn't match the language used in \Cref{lem:long_0_and_1_not_pattern_Sturmian}, either the old or my new after rewriting. Should probably make look similar.}
This means that $1^{\tau_n}$ is a prefix of at least $2^n$ words in $L_{x'}(\tau_{n+1})$ other than $1^{\tau_{n+1}}$. 

Putting this together yields $|L_{x'}(\tau_{n+1})| \geq 2|L_{x'}(\tau_n)| + 2^n - 1$. A simple proof by induction then yields the inequality $|L_{x'}(\tau_n)| \geq (n+2) 2^{n-1}$ for all $n$.

Then, if one defines the window $\rho_n = \tau_n + [0,N)$, clearly $|L_x(\rho_n)| \geq |L_{x'}(\tau_n)|$, since $\phi$ is a surjection from the former set to the latter. Since $|\rho_n| \leq N 2^n$ and pattern complexity is monotone, we arrive at $p^*_x(N 2^n) \geq (n+1) 2^{n-1}$ for all $n$. For arbitrary $k$, we can take the maximal $n$ for which $N 2^n \leq k$ to get
\[
    p^*_x(k) \geq p^*_x(N 2^n) \geq (n+1) 2^{n-1} \geq k \frac{n+1}{4N} \geq \frac{k \log_2 (k/N)}{4N},
\] 
yielding 
\[
    \liminf_{k \to \infty} \frac{p^*_x(k)}{k \ln k} \geq \frac{1}{4N\ln 2}
\]
and completing the proof.
\end{proof}

%\Anote{In above proof, I change the minimal subshift $M$ to $Y$ and correct notation for language $L_X(N)$ instead of $L_N(X)$.}

% \rp{A general thing we/I need to work on is that there's a subtle step/issue in our final proof steps here. We prove that for any uniformly recurrent pattern Sturmian/linear pattern complexity $x$, the orbit closure of $x$ is a minimal subshift $X$ which contains a point $x'$ which is obtained by coding a rotation of a (virtual) circle/odometer. It's immediate that $x'$ is a coding sequence as we've defined them earlier. But we need to show that $x$ is as well. This follows pretty quickly from our definitions I think, but needs to be shown. Either we could do an ad hoc argument for each group, or there might be a simple general statement that EVERY point in a subshift with a certain MEF partition still comes from coding orbits, but with ambiguity at the boundary points. But since the rotations are always irrational, no boundary point gets hit more than once, so this should still fall into our definitions.  I can think about how to best deal with this.}

In light of \cref{mainthm:recnotmin}, we now know that $x$ is uniformly recurrent and $X = \overline{\OO(x)}$ is a minimal subshift with non-superlinear maximal pattern complexity, implying by \Cref{thm:possible_MEF_nonsuperlinear} that the MEF of $X$ is either the product of a circle and a finite cyclic group or an odometer. %\Anote{We need to assume $X$ is not periodic here.} 
We begin with the first case.

\begin{proposition}\label{part1}
If $x$ is a recurrent sequence with non-superlinear maximal pattern complexity and if $X = \overline{\OO(x)}$ has MEF a product of a circle with finite cyclic group, then $x$ is a periodic interleaving of finitely many sequences which are each either a circle rotation interval coding sequence or constant, and all of the circle rotation interval coding sequences are associated with the same irrational rotation. 
\end{proposition}
% In other words, there exist $z \in \mathbb{T}$, $k \in \N$, $\alpha \in \mathbb{Q}^c$, $n_i \in \N$, and $z_{i,j} \in \mathbb{T}$ with $1 \leq i \leq k, 1 \leq j \leq 2n_i$ with the following properties:

%\begin{itemize}
%\item For each $1 \leq i \leq k$, $z_{i,1} < z_{i,2} < \cdots < z_{i,2n_i}$
%\item The orbit of $(z,1)$ under addition by $(\alpha, 1)$ on the group $\mathbb{T} \times \faktor{\Z}{k\Z}$ contains no $z_{i,j}$
%\item The orbit coding $
%\end{itemize}

\begin{proof}
By \Cref{mainthm:recnotmin}, $x$ is uniformly recurrent and $X$ is minimal; denote its MEF by $\left( \mathbb{T} \times \faktor{\Z}{k\Z}, +(\alpha, 1)\right)$, and denote this group by $G$. Denote the associated MEF partition by $\{U_0, U_1, B\}$.  

Then by \Cref{gencoding}, there exists a partition $\{B_0, B_1\}$ of $B$ and $g \in G$ so that the orbit of the identity is coded by $x$ in the sense that $x(n) = i$ if and only if $n(\alpha, 1) \in (U_i \cup B_i) - g$. We note that for each $j$, the sets $(U_i \cup B_i) - g$ induce a partition $\{A^{(j)}_0, A^{(j)}_1\}$ of $\mathbb{T} \times \{j\}$, which must be either trivial (i.e. one of the sets is empty) or into finitely many intervals (since all endpoints are elements of $B$). 

Split $x$ into the sequences $x^{(j)}$, $0 \leq j < k$, defined by $x^{(j)}(n) = x(j + nk)$. Then for each $j$, the sequence $x^{(j)}$ is obtained by coding the orbit of $j \alpha$ under rotation by $k \alpha$ by the partition $\{A^{(j)}_0, A^{(j)}_1\}$ since
\begin{align*}
    x^{(j)}(m) = x(j + mk) = i &\Longleftrightarrow (j+mk)(\alpha,1) = (j\alpha + m(k\alpha), j) \in (U_i \cup B_i - g) \\
    &\Longleftrightarrow j\alpha + m(k \alpha) \in A^{(j)}_i.
\end{align*}
If the partition $\{A^{(j)}_0, A^{(j)}_1\}$ is trivial, then $x^{(j)}$ is constant, and if it consists of finitely many intervals, then $x^{(j)}$ is a circle rotation interval coding sequence by definition. This completes the proof.
\end{proof}

It remains only to treat the case where the MEF of $X$ is an odometer.

\begin{proposition}\label{part2}
If $x$ is a recurrent sequence with non-superlinear maximal pattern complexity and if $X = \overline{\OO(x)}$ has MEF an odometer, then $x$ is an element of an $m$-hole Toeplitz subshift for some $m > 0$. %\rp{Ugh, there is an annoying issue here; $x$ itself need not be Toeplitz. Literally speaking, being Toeplitz requires that all of $\N$ is partitioned by arithmetic progressions, and Toeplitz subshifts contain non-Toeplitz sequences; these are the ones where the residue classes of some hole doesn't go to infinity. This isn't difficult, we just have to think about how to phrase correctly.}
\end{proposition}

\begin{proof}
Suppose that $x$ is as in the theorem. As before, $X$ must be minimal and denote by $\mathcal{O} = \varprojlim \faktor{\Z}{n_k\Z}$ its MEF. Then since $X$ is an almost $1$-$1$ extension of $\mathcal{O}$, it is a Toeplitz subshift with period structure $(n_k)$. By \Cref{grpcor}, there is an associated MEF partition with $|B| = m < \infty$, which implies by \Cref{mhole} that $X$ is an $m$-hole Toeplitz subshift.

%Then it is Toeplitz since it is an almost $1$-$1$ extension of an odometer. Suppose that $x \in X$ is a Toeplitz sequence defined by period structure $n_1 | n_2 | n_3 | \cdots$ and disjoint sets $S^{(0)}_k, S^{(1)}_k \subseteq \faktor{\Z}{n_k \Z}$ (which may be empty). In other words, $S^{(j)}_k \subseteq S^{(j)}_{k+1} \pmod{n_k}$ for all $k \in \N$, $j \in \{0,1\}$, $\bigcup_k (S^{(0)}_k \cup S^{(1)}_k) = \N$, and $x(n) = i$ if and only if $n \in S^{(i)}_k$ for some $k$. 

%We assume without loss of generality that each $S^{(i)}_k$ is maximal, i.e. that for every $k$ and $n \in \faktor{\Z}{n_k \Z} - (S^{(0)}_k \cup S^{(1)}_k)$, $x(n + n_k \N)$ contains both $0$s and $1$s. Since $X$ is infinite, $\faktor{\Z}{n_k \Z} - (S^{(0)}_k \cup S^{(1)}_k)$ is nonempty for each $k$.

%This system has the properties of \Cref{coding} for $(Y, T)$ the $(n_k)$-odometer, $U_i$ the set of all $y$ such that $y \pmod{n_k} \in S^{(i)}_k$ for some $k$, and $B$ the set of all $y$ so that $y \pmod{n_k} \in \faktor{\Z}{n_k \Z} - (S^{(0)}_k \cup S^{(1)}_k)$ for all $k$. Since $X$ has non-superlinear maximal pattern complexity, by \Cref{grpcor} it must be the case that $B$ is finite, which means that there is a uniform upper bound on the cardinalities of $\faktor{\Z}{n_k \Z} - (S^{(0)}_k \cup S^{(1)}_k)$, i.e. $X$ is a Toeplitz with finitely many holes.

\end{proof}

\begin{remark}
We note that although the converse of \Cref{part1} holds, i.e. all such codings of circle rotations have linear maximal pattern complexity, it is very much false for \Cref{part2}; there exist $1$-hole Toeplitz sequences which are not even null (see \cite[Section 11]{Kerr_Li-independence_Cstar}) and so have $p^*_x(n) = 2^n$.
\end{remark}

\begin{proof}[Proof of \cref{mainthm:non-superlinear}]
The theorem follows by combining Propositions \ref{part1}, \ref{part2} and \Cref{mainthm:recnotmin}.
\end{proof}

%\section{Non-superlinear complexity: non-minimal case and proof of \cref{mainthm:non-superlinear} and \cref{mainthm:recnotmin}}

%We first prove \Cref{mainthm:recnotmin}, which shows that the non-superlinear case cannot even occur for recurrent sequences which are not uniformly recurrent.

%\begin{proof}[Proof of \cref{mainthm:non-superlinear}]
%The theorem follows by combining Propositions \ref{part1}, \ref{part2} and \ref{rec}.    

%If $X$ is a recurrent subshift such that $\liminf_{n \to \infty} p_X^*(n)/n < \infty$, then by \cref{mainthm:recnotmin}, $X$ is minimal. Then Propositions \ref{part1}, \ref{part2} imply that $X$ is either an interleaving of finitely many codings of the same irrational circle rotation by finitely many intervals or a Toeplitz shift with finitely many holes. 
%\end{proof}

%Since maximal pattern complexity is obviously decreasing upon passing to a subset, we have the following simple corollary which will be needed for the proof of \cref{mainthm:sturmian-nonrecurrent} later.

%\begin{corollary}\label{cor:rec-unif}
%If $X$ has non-superlinear maximal pattern complexity, then every recurrent point in $X$ is uniformly recurrent.
%\end{corollary}

%\begin{proof}
%By \Cref{mainthm:recnotmin}, every recurrent $x \in X$ which is not eventually periodic must be uniformly recurrent. In addition, any recurrent sequence which is eventually periodic must be periodic, and again uniformly recurrent.
%\end{proof}

\section{Pattern Sturmian: minimal case and proof of \texorpdfstring{\cref{mainthm:sturmian}}{Theorem A}}

We now restrict to the case where $x$ is pattern Sturmian (recall that this means $p^*_x(n) = 2n$ for all $n$), and recurrent, therefore uniformly recurrent by \Cref{mainthm:recnotmin}. %These are codings of circle rotations by two intervals and Toeplitz subshifts which are simple or very close to simple (STATE FORMALLY? REFERENCE RELEVANT KAMAE ET AL PAPERS \rp{I'll do this soon}). 
We will show that $x$ must be either a circle rotation interval coding sequence or a sequence in a nearly simple Toeplitz subshift.

Our main tool is the fact that by \Cref{main}, \Cref{lem:boundary_non_empty}, and \Cref{thm:possible_MEF_nonsuperlinear}, the minimal subshift $X = \overline{\OO(x)}$ has an associated MEF partition with $G$ either the product of a circle and a finite cyclic group or an odometer, and $|B|$ either $1$ or $2$. 
%This leads to an immediate resolution of the case where $G$ is virtually a circle.
% \begin{definition}
% Let $\alpha$ be an irrational number and $I$ be a nonempty propert interval of $\T$. Then the sequence $x_k = 1_{I}(k \alpha)$ for $k \in \Z$ is called a \emph{coding of an irrational rotation by two intervals}.
% \end{definition}

We again begin with the case where $G$ is the product of a circle and a finite cyclic group. The following lemma is needed for \cref{circcase} and \cref{lem:glue_up}.
\begin{lemma}\label{lem:sturmian_no_constant_words}
Let $I$ be a nonempty, proper interval on $\T$ and $\alpha$ be an irrational number. Let $x(n) = 1_I(n \alpha)$ for all $n \in \N_0$.
There exists $k \in \N$ such that for sufficiently large $n$, the window $\tau = \{0, k, 2k, \ldots, (n-1)k\}$ satisfies $|L_{x}(\tau)| = 2n$ and $L_x(\tau)$ does not contain the constant words $00\ldots 0$ and $11 \ldots 1$.

%\Anote{We will need the following strengthening of \cref{lem:sturmian_no_constant_words} for \cref{lem:glue_up}: There exists $k \in \N$ such that for sufficiently large $n$, the window $\tau = \{0, k, 2k, \ldots, (n-1)k\}$ satisfies $|F_{x}(\tau)| = 2n$ and $F_x(\tau)$ does not contain the constant words $00\ldots 0$ and $11 \ldots 1$. This strengthening should be already in the proof. I'll check it tomorrow.}
\end{lemma}
\begin{proof}
%\Anote{Need a proof...} 

%\rp{I agree with this being a separate lemma, but I might be able to shorten the proof; I'll think about it.}] \Anote{We need the full strength of \cref{lem:sturmian_no_constant_words} for \cref{lem:glue_up} below.}

%Assume $x(n) = 1_I(n \alpha)$ where $I \subseteq \T$ is an interval and $\alpha$ is an irrational number. 
Suppose $I = (a, b), [a, b), (a, b]$, or $[a, b]$. Without loss of generality, assume $0 < b - a \leq 1/2$. Denote the interval $I_1 = I$ and $I_0 = [0, 1) \setminus I$. Choose $k$ so that $\theta = k \alpha \bmod 1$ satisfies 
\begin{equation}\label{eq:theta_min}
    \theta < \min \{|I_0|, |I_1|\}.
\end{equation} 
Additionally, choose $k$ so that 
\begin{equation}\label{eq:no_integer}
    \text{there is no integer $m$ for which $b - a = m\theta$ or $a - b = m \theta \bmod 1$.}
\end{equation} 
(This is possible since if $b - a = m_1 k_1 \theta = m_2 k_2 \alpha \bmod 1$, then since $\alpha$ is irrational, $m_1 k_1 = m_2 k_2$. Then we just need to choose $k$ sufficiently large.)

Let $n$ be large enough so that the gaps between adjacent elements of $\{0, \theta, \ldots, (n-1) \theta\}$ in $\T$ is less than $\min\{|I_0|, |I_1|\}$. Let $\tau$ be the $n$-window $\{0, k, 2k, \ldots, (n-1) k\}$. We will show that $\tau$ satisfies the conclusion of our lemma. 

For every $t \in \T$, 
\[
    (t, t + k \alpha, \ldots, t + (n-1)k \alpha) = (t, t + \theta, \ldots, t + (n-1) \theta). 
\]
By the denseness of $\{0, \theta, \ldots, (n-1) \theta\}$, for every $t \in \T$, at least one of the elements $t, t + \theta, \ldots, t + (n-1) \theta$ belongs to $I_0$ and one belongs to $I_1$. It follows that the constant word $00\ldots0$ and $11 \ldots 1$ do not belong to $L_{x}(\tau)$.

It remains to show that $|L_x(\tau)| = 2n$.
Now if 
\[
    m \theta \in (I_{c_0}) \cap (I_{c_1} - \theta) \cap \cdots \cap (I_{c_{n-1}} - (n-1) \theta),
\]
then 
\[
    x(mk + \tau) = c_0 c_1 \ldots c_{n-1}.
\]
Since $\{m\theta: m \in \N\}$ is dense in $\T$, $|L_x(\tau)|$ is equal to the number of tuple $(c_0, \ldots, c_{n-1})$ so that the intersection 
\begin{equation}\label{eq:intersection}
     (I_{c_0}) \cap (I_{c_1} - \theta) \cap \cdots \cap (I_{c_{n-1}} - (n-1) \theta)
\end{equation}
is nonempty.

By \eqref{eq:theta_min}, each nonempty intersection in \eqref{eq:intersection} is connected. This can be proved by induction on $n$ and using the fact that $I - (n-1) \theta$ and $I - n \theta$ always intersects.

Note that $\mathcal{I} = \{I_0, I_1\}$ is a partition of the circle $\T$.
By \eqref{eq:no_integer}, no two endpoints of $\mathcal{I} - i\theta$ coincide and so the partition 
\begin{equation}\label{eq:partition_big}
    \mathcal{I} \vee (\mathcal{I} - \theta) \vee \cdots \vee (\mathcal{I} - (n-1) \theta)
\end{equation}
where ``$\vee$'' means the common refinement of the partitions consists of exactly $2n$ subintervals. As discussed before, each interval produces a unique $n$ length word in $L_x(\tau)$ and we are done.
\end{proof}

\begin{proposition}\label{circcase}
Let $X$ be a minimal pattern Sturmian subshift and let $G$ be its MEF. If $G$ is the product of a circle and a finite cyclic group, then $G$ is the circle and there is an associated MEF partition with $|B| = 2$.
\end{proposition}
\begin{proof}
It is immediate that the associated partition $G = U_0 \cup U_1 \cup B$ must have $|B| = 2$ since a circle can only be disconnected by removing at least two points. The only case which must be ruled out is 
$G = \mathbb{T} \times \faktor{\Z}{k\Z}$ for some $k \geq 2$. Assume for a contradiction
that $G$ is of this form, and without loss of generality that $B = \{(a, 0), (b, 0)\}$ for some $a \neq b \in \mathbb{T}$ and $\mathbb{T} \times \{1\} \subseteq U_1$.

Suppose the rotation on $G$ is $(\alpha, 1)$ where $\alpha \in \T$ is irrational and $1 \in \faktor{\Z}{k\Z}$. The subsequence $x' = (x(kn))_{n \in \mathbb{Z}}$ is then a coding of the rotation by $\alpha$ on the circle $\T \times \{0\}$. By \cref{lem:sturmian_no_constant_words}, there exists $n$ and an $n$-window $\tau$ such that $|L_{x'}(\tau)| = 2n$ and $L_{x'}(\tau)$ does not contain the constant words $00 \ldots 0$ and $11 \ldots 1$. 

Now, consider the window $\tau' = k\tau$. By considering shifts by multiples of $k$, we have $L_{x'}(\tau) \subseteq L_x(\tau')$. In addition, since all coordinate in the shifted window $\tau' + 1$ is congruent to $1$ mod $k$, $x(\tau' + 1)$ is the constant word $11 \ldots 1$. It follows that $|L_x(\tau')| \geq 2n + 1$, contradicting the assumption that $x$ is pattern Sturmian.
\end{proof}

We now approach the case of $G$ an odometer, which is more difficult since it can be disconnected by boundary sets of cardinality $1$ or $2$. We will, however, prove that if $X$ is pattern Sturmian, then there must exist an MEF partition with $|B| = 1$, which implies that $X$ is a $1$-hole Toeplitz by \Cref{mhole}. Then, \Cref{onehole} implies that $X$ is nearly simple Toeplitz. %\Anote{Just mark here so we don't forget. What we get is $x$ is in the orbit closure of a nearly simple Toeplitz sequence.}

%Luckily, the following result of \cite{Gjini_Kamae_Bo_Yu-Mei-toeplitz} completely characterizes pattern Sturmian $1$-hole Toeplitz sequences. We do not give a definition of the technical condition $D((w?)^\infty, L) \leq 0$ here; see \cite[Page 1076]{Gjini_Kamae_Bo_Yu-Mei-toeplitz} for definitions.

%\begin{theorem}[{\cite[Theorem 2]{Gjini_Kamae_Bo_Yu-Mei-toeplitz}}]
%\label{onehole}
%If $x$ is a $1$-hole Toeplitz, it is pattern Sturmian if and only if it is either simple Toeplitz or it is the image of a simple Toeplitz under a single constant-length morphism $\alpha: \{0,1\} \rightarrow \{0,1\}^C$ of the form $0 \mapsto w0, 1\mapsto w1$ satisfying the condition $D((w?)^\infty, L) \leq 0$ for $L \subseteq \{0, \ldots, |w|\}$.
%\end{theorem}

We first need the following lemma about an explicit maximal $3$-window for simple $1$-hole Toeplitz sequences.
%\rp{I'm currently trying to change this lemma so that it takes care of the later nonrecurrent Toeplitz case as well.}

%\rp{check statement of this lemma; should it be for subshifts?}

\begin{lemma}\label{1hole3lang}
If $x$ is a simple $1$-hole Toeplitz defined by odometer with period structure $(n_k)$ and sequence $a_k \in \{0,1\}$ 
(meaning that all residue classes mod $n_k$ except one are filled with the letter $a_k$), and 
$c < d < e < f < g < h$ are chosen with $a_c = a_e = a_g = 0$ and $a_d = a_f = a_h = 1$, then the window 
$\tau = \{0, n_e, n_f\}$ is maximal, and 
\[
L_x(\tau) = \{000, 001, 010, 100, 101, 111\}.
\]
\end{lemma}

\begin{proof}
Assume that $x$ is such a sequence and that $c < d < e < f < g < h$ are chosen as above. 
By shifting $x$, we may assume that $0$ is the nonconstant residue class $\pmod{n_k}$ for $k \leq h$. In other words, making the notation $x(i, n) := x(i + n \mathbb{N})$, for all $k \leq h$, $x(i, n_k)$ is a constant sequence of $a_k$ unless $i = 0$. Phrased slightly differently, for $k \leq h$, if $m \in \mathbb{N}$ is a multiple of $n_k$ but not 
$n_{k+1}$, then $x(m) = a_k$. We note that this implies that for any $m = \pm n_{i_0} \pm n_{i_1} \pm \cdots \pm n_{i_j}$ for 
$i_0 < i_1 < \cdots < i_j \leq h$, $x(m) = a_{i_0}$.
 
Define the window $\tau = \{0, n_e, n_f\}$; we exhibit the claimed words via shifts of $\tau$.

\begin{itemize}
\item $x(\tau + n_c) = x(n_c, n_c + n_e, n_c + n_f) = a_c a_c a_c = 000$.
\item $x(\tau + n_d) = x(n_d, n_d + n_e, n_d + n_f) = a_d a_d a_d = 111$.
\item $x(\tau + n_f - n_e) = x(-n_e + n_f, n_f, -n_e + 2n_f) = a_e a_f a_e = 010$.
\item $x(\tau + n_g) = x(n_g, n_e + n_g, n_f + n_g) = a_g a_e a_f = 001$.
\item $x(\tau + n_g - n_f) = x(-n_f + n_g, n_e - n_f + n_g, n_g) = a_f a_e a_g = 100$.
\item $x(\tau + n_h - n_f) = x(-n_f + n_h, n_e - n_f + n_h, n_h) = a_f a_e a_h = 101$.
\end{itemize}

Since simple Toeplitz sequences are pattern Sturmian, $p^*_x(3) = 6$ and so these are all of the words in $L_x(\tau)$.
\end{proof}

\begin{proposition}\label{b1}
If $X$ is minimal pattern Sturmian with MEF an odometer, then there exists an associated MEF partition with $|B| = 1$.
\end{proposition}

%\rp{Think about this proof; do we need to include non-Toeplitz sequences?}
%\Anote{I slightly don't like using $\mathcal{O}$ for odometer MEF. It looks like the number zero ... We can use $G$ as in the case of general MEF and circle or find a different symbol}

\begin{proof}
Assume that $X$ is minimal and pattern Sturmian, with the MEF being the odometer $\mathcal{O} = \lim_{\longleftarrow} \faktor{\Z}{n_k \Z}$. Since $X$ is an almost $1$-$1$ extension of $\mathcal{O}$, it is a Toeplitz subshift with period structure $(n_k)$.
By \Cref{main} and \Cref{lem:boundary_non_empty}, $(X, \sigma)$ has an associated MEF partition with boundary set $B$ of cardinality $1$ or $2$.

%The case $|B| = 0$ is impossible; it would require $U_0, U_1$ clopen, meaning that membership depends on only finitely many coordinates and thus the coding sequence $x$ would be periodic. 

If $|B| = 1$, the proof is complete. So the remaining case is $|B| = 2$, say $B = \{(i_k), (j_k)\} \subseteq \mathcal{O}$. Without loss of generality, we assume that $i_0 \neq j_0$ by truncating $(n_k)$ if necessary. We also note that $(i_k), (j_k)$ are not in the orbit of $0$ in $\mathcal{O}$ by definition of MEF partition, and so $i_k, j_k \rightarrow \infty$.
Then, the coding sequence $x$ from \Cref{main} is a $2$-hole Toeplitz; for each $k$ and $0 \leq i < n_k$, $x(i, n_k) := x(i + n_k \mathbb{N}_0)$ is constant if and only if $i \notin \{i_k, j_k\}$. 

Now, $x' = x(i_0, n_0)$ and $x'' = x(j_0, n_0)$ are both $1$-hole Toeplitz sequences with period structure $(n_k/n_0)$, and for any window $\tau$, $L_x(n_0 \tau)$ contains $L_{x'}(\tau) \cup L_{x''}(\tau)$. This immediately implies that both $x'$ and $x''$ are pattern Sturmian, and so by \Cref{onehole}, each is either a simple Toeplitz or has a decomposition into residue classes where one residue class is a simple Toeplitz and other residues are constant. %\Anote{Theorem 2.12 does not say so explicitly. Maybe we can add this explanation right after statement of Theorem 2.12} 
In fact, the proof of \Cref{onehole} in \cite{Gjini_Kamae_Bo_Yu-Mei-toeplitz} %(\Anote{Theorem 2 in that paper}) 
shows that this decomposition can always be taken modulo the first period in the period structure, which for $x, x''$ is $n_1/n_0$.
%\rp{Make sure of this!}
Therefore, by truncating the first term from $(n_k)$, we may assume without loss of generality that $x'$ and $x''$ are simple.

Say that $x'$ and $x''$ are defined by the sequences of letters $(a_k), (b_k)$ respectively. If there exist infinitely many $k$ for which $a_k \neq b_k$, then we may assume without loss of generality that there are infinitely many $e$ for which $a_e = 0$ and $b_e = 1$. Since each sequence takes values $0,1$ infinitely often, by taking $e$ large enough we may then choose 
$c < c' < d < d' < e < f < g < g' < h < h'$ so that
\begin{itemize}
\item $a_c = a_e = a_g = 0$ 
\item $a_d = a_f = a_h = 1$
\item $b_{c'} = b_e = b_{g'} = 1$
\item $b_{d'} = b_{h'} = 0$.
\end{itemize}
Then, we define $\tau = \{0, n_e/n_0, n_f/n_0\}$ (recalling that the period structure for $x'$ and $x''$ is $(n_k/n_0)$). By applying \Cref{1hole3lang} to $x'$ and $c < d < e < f < g < h$, we get 
$L_{x'}(\tau) = \{000, 111, 001, 010, 100, 101\}$. If $b_f = 0$, then applying \Cref{1hole3lang} to the bit flip of $x''$
and $c' < d' < e < f < g' < h'$ yields $L_{x''}(\tau) = \{000, 111, 110, 101, 011, 010\}$. %\Anote{Why bit flip of $x''$ in x?}
Then $L_x(n_0 \tau) \supseteq L_{x'}(\tau) \cup L_{x''}(\tau) = \{0,1\}^3$, and so $p^*_x(3) = 8$, a contradiction.
If instead $b_f = 1$, then the same argument from the last bullet point in the proof of \Cref{1hole3lang} shows that
$L_{x''}(\tau)$ contains $b_f b_e b_{h'} = 110$. Then again $L_x(n_0 \tau) \supseteq L_{x'}(\tau) \cup L_{x''}(\tau)$ has size greater than $6$, a contradiction.

We may therefore assume that $a_k = b_k$ for sufficiently large $k$, and by passing to a subsequence of $(n_k)$ if necessary, that $a_k = b_k$ for all $k$ and alternates between $0$ and $1$, i.e. either $a_k = b_k = k \pmod 2$ for all $k$ or $a_k = b_k = k + 1\pmod 2$ for all $k$.

%Clearly we may assume WLOG that $i_k = 0$ for all $k$ and that $j_k \leq n_k/2$ by shifting and taking limits.

We assume for a contradiction that there exists $k$ where $j_k - i_k \neq n_k/2$. Without loss of generality, we can then assume (by shifting, truncating $(n_k)$, and possibly switching $i_k$ and $j_k$) that $i_0 = i_1 = i_2 = 0$ and $j_0 < n_0/2$.
If we define $g = \gcd(j_0, n_0) \leq j_0 < n_0/2$, then the sequence $y := x(0, g)$ is still a $2$-hole Toeplitz sequence, with period structure $(n_k/g)$ and two nonconstant residue classes $0, j_k/g$ for each $k$. 
In addition, $y$ is nonperiodic and $L_x(g\tau) \supseteq L_y(\tau)$ for all $\tau$, so $y$ is pattern Sturmian. We may then replace 
$x, (n_k), (j_k)$ with $y, (n_k/g), (j_k/g)$, and so assume without loss of generality that $g = \gcd(j_0, n_0) = 1$.

The key to the rest of our proof is the following observation: if, for any $k$, there exists $r \in (0, n_k), r \neq \pm(j_k - i_k)$, so that $x(i_k + r, n_k) = a^{\N_0}$ and $x(j_k + r, n_k) = \overline{a}^{\N_0}$ are different constant sequences, then $x$ is not pattern Sturmian, yielding a contradiction. To see this, we first note that $x'_k := x(i_k, n_k)$ and $x''_k := x(j_k, n_k)$ are simple $1$-hole Toeplitzes with the same period structure and generating letters, and so have the same language. We may then consider a maximal $2$-window $\tau$ for $x'_k$ and $x''_k$, meaning that $|L_{x'_k}(\tau)| = |L_{x''_k}(\tau)| = 4$. Now, if $r$ as above exists, define the window $\tau' := n_k \tau + r$. By considering shifts in $n_k \mathbb{N}_0 + i_k$, we see that $L_x(\tau')$ contains all words in $L_{x'_k}(\tau)$ followed by $a$, and by considering shifts in $n_k \mathbb{N}_0 + j_k$, we see that $L_x(\tau')$ contains all words in $L_{x''_k}(\tau)$ followed by $\overline{a}$. Therefore, 
$|L_x(\tau')| = 8$, contradicting the fact that $x$ is pattern Sturmian.

Therefore, for all $k$, such $r$ does not exist. Since $0 < j_0 < n_0/2$, we know $n_0 > 2$ and $2j_0 \notin \{0,j_0\} \pmod{n_0}$. Therefore, $x(2j_0, n_0)$ is a constant sequence, say, without loss of generality, of all $0$s. For any $2 < m < n_0$, since $\gcd(j_0, n_0) = 1$, 
$(m-1)j_0 \not\equiv \pm j_0 \pmod{n_0}$, and so by taking $r = (m-1)j_0$, we see that $x(i_0 + r, n_0) = x((m-1)j_0, n_0)$ and 
$x(j_0 + r, n_0) = x(mj_0, n_0)$ are the same constant sequence, i.e. all $0$s. Therefore, since $\gcd(j_0, n_0) = 1$, all residue classes $\pmod{n_0}$ except $0$ and $j_0$ are $0$s.

If $a_1 = b_1 = 1$, then all $x(m, n_1)$ with $m \in \{0, j_0\} \pmod{n_0}$ but $m \notin \{0, j_1\} \pmod{n_1}$ are constant sequences of $1$s. We then define $r = n_0 - j_1$. First, $x(r + i_1, n_1) = x(n_0 - j_1, n_1)$ is all $0$s since $n_0 - j_1 \equiv -j_0 \notin \{0, j_0\} \pmod{n_0}$. And $x(r + j_1, n_1) = x(n_0, n_1)$ is all $1$s since $n_0 \equiv 0 \pmod{n_0}$ but $n_0 \notin \{0, j_1\} \pmod{n_1}$. This again yields a contradiction.

Finally, if $a_1 = b_1 = 0$, then $a_2 = b_2 = 1$, and by definition all $x(m, n_2)$ with $m \notin \{0, j_1\} \pmod{n_1}$ are all $0$s and all $x(m,n_2)$ with $m \in \{0, j_1\} \pmod{n_1}$ but $m \notin \{0, j_2\} \pmod{n_2}$ are constant sequences of $1$s.
We then define $r = n_1 - j_2$. First, $x(r + i_2, n_2) = x(n_1 - j_2, n_2)$ is all $0$s since $n_1 - j_2 \equiv -j_0 \notin \{0, j_0\} \pmod{n_0}$. And $x(r + j_2, n_2) = x(n_1, n_2)$ is all $1$s since $n_1 \equiv 0 \pmod{n_1}$ but $n_1 \notin \{0, j_2\} \pmod{n_2}$. This yields our final contradiction, meaning that our original assumption that $j_k - i_k \neq n_k/2$ for some $k$ is false.

Therefore, $j_k - i_k = n_k/2$ for all $k$. In addition, for each $0 < r < n_k/2$, $x(i_k + r, n_k)$ and $x(j_k + r, n_k)$ must be the same constant sequence, meaning that in fact $x$ is constant on all residue classes $\pmod{n_k/2}$ except $i_k$ for all $k$. Therefore, in fact $x$ is a $1$-hole Toeplitz with period structure $(n_k/2)$ and nonconstant residue classes $(i_k)$. This means that $x$ comes from coding the MEF partition of the associated odometer $\mathcal{O}'$ where $U_i$ consists of the disjoint union of all clopen sets coming from residues $m \pmod{n_k}$ for which $x(m, n_k)$ is a constant sequence of all $i$s. This partition has $B = \{(i_k)\}$, and since $|B| = 1$, the proof is complete.
\end{proof}

% This yields the following complete characterization.

% \begin{theorem}\label{thm:minimal_pattern-Sturmian-classification}
% Any minimal pattern Sturmian subshift is either a coding of an irrational circle rotation by two intervals or a $1$-hole Toeplitz subshift as in \Cref{onehole}.
% \end{theorem}

% \begin{proof}
% This follows immediately from \Cref{circcase}, \Cref{onehole}, and \Cref{b1}.
% \end{proof}

\begin{proof}[Proof of \cref{mainthm:sturmian}]

If $x$ is a recurrent pattern Sturmian sequence, then by \cref{mainthm:recnotmin}, \cref{thm:possible_MEF_nonsuperlinear}, and \cref{circcase}, $X = \overline{\OO(x)}$ is minimal and its MEF is either the circle or an odometer. If the MEF is the circle, then \Cref{gencoding} and \Cref{circcase} together imply that $x$ is a simple circle rotation coding sequence. If the MEF is an odometer, then $X$ is Toeplitz and by Theorems \ref{onehole}, \ref{mhole}, and \ref{b1}, $X$ is an nearly simple $1$-hole Toeplitz.  
%\Anote{Need a short explanation why MEF being odometer and boundary size $1$ in the associated partition implies $X$ is an $1$-hole Toeplitz.} \rp{probably we should have a more general discussion of Toeplitz/boundary set/holes earlier in the paper, maybe in a preliminaries section.}
\end{proof}

\section{Pattern Sturmian: nonrecurrent case and proof of \texorpdfstring{\cref{mainthm:sturmian-nonrecurrent}}{Theorem B}}

Finally, we wish to characterize nonrecurrent pattern Sturmian sequences. Suppose that $x$ is such a sequence. Then it is not uniformly recurrent, but its orbit closure $X$ contains a uniformly recurrent sequence $y$, which must itself be periodic or pattern Sturmian since $p^*_y(n) \leq p^*_x(n) = 2n$. 
%\Anote{food for thought: Should we keep ``$m$'' as minimal sequence in $X$ or use ``$y$'' instead? ``m'' is usually used for integer}
If $y$ is pattern Sturmian, then by \cref{mainthm:sturmian}, it is either a simple circle rotation coding sequence or it is in a nearly simple Toeplitz subshift, in which case we can assume without loss of generality that $y$ is a nearly simple Toeplitz sequence. Our proof of \cref{mainthm:sturmian-nonrecurrent} will consist of three pieces: 

\begin{enumerate}
    \item $y$ cannot be a nearly simple Toeplitz sequence;
    
    \item if $y$ is a simple circle rotation coding sequence, then $x$ is a nonrecurrent simple circle rotation coding sequence;
    
    \item if $y$ is periodic, then it is constant and $x$ is almost constant.
\end{enumerate}

We begin with the first result.

\begin{theorem}\label{thm:contain_minimal_toeplitz_cannot}
    If a pattern Sturmian sequence $x$ has orbit closure $X$ containing $y$ which is nearly simple Toeplitz, then $x$ is recurrent.
\end{theorem}

\begin{proof}
Suppose that $x$ is pattern Sturmian and has orbit closure $X$ containing nearly simple Toeplitz $y$ with period structure $(n_k)$, and without loss of generality suppose that all residue classes of 
$y \pmod{n_1}$ are constant except for the single $y(i_1 + n_1 \mathbb{N})$ which is simple Toeplitz. Now, fix any $k \geq 2$. Since the orbit closure of $x$ contains $y$, $x$ contains arbitrarily long words on which all but one residue class $\pmod{n_k}$ are constant and equal to the same letter $a_k$ and the remaining residue class is a subword of the simple Toeplitz sequence $m^{(k)} := y(i_k + n_k \mathbb{N})$. Let's begin with the case 
$a_k = 1$. Then by \Cref{1hole3lang}, there exists a $3$-window $\tau_k$ so that $L_{y^{(k)}}(\tau_k) = 
\{000, 111, 001, 010, 100, 101\}$. Since $y^{(k)}$ is uniformly recurrent, there exists $N_k > \textrm{diam}(\tau_k)$ so that every word in $L_{y^{(k)}}$ of length $N_k$ contains all six words in $L_{y^{(k)}}(\tau_k)$. 

Let's say specifically that $p_k < q_k$ are chosen such that $q_k - p_k > n_k N_k$ and that
$x(i) = 0$ if $p_k \leq i < q_k$ and $i \neq i'_k \pmod{n_k}$, and $w_k := x([p_k, q_k) \cap (i'_k + n_k \mathbb{N})$ is a subword of $y^{(k)}$. Then since $w_k$ has length at least $N_k$, it contains $000, 111, 001, 010, 100, 101$ as $\tau_k$-subwords. This means that $L_x(n_k \tau_k)$ contains these words as well, and so cannot contain $011$ or $110$ by the assumption that $x$ is pattern Sturmian. For each $i \in [0, n_k), i \neq i'_k$, $x(i + n_k \mathbb{N})$ contains $N_k > \textrm{diam}(\tau_k)$ consecutive $1$s. If this sequence contained a $0$, then without loss of generality, we could consider it the nearest $0$ to the consecutive $1$s, and we would have either $011, 110$ as a $n_k \tau_k$-subword, a contradiction. Therefore, for each such $i$, $x(i + n_k \mathbb{N})$ is constant of all $1$s. 

If instead $a_k = 0$, we would apply \Cref{1hole3lang} to the bit flip of $y^{(k)}$ to get a $3$-window $\tau_k$ with $L_{y^{(k)}}(\tau_k) = \{000, 111, 011, 101, 110, 010\}$, and the same argument shows that all residue classes of $x$ except one are constant sequences of $0$s. 

In either event, we have shown that for all $k \geq 2$, all residue classes of $x$ except one are constant and equal to 
$a_k$. The same proof, applied to $k = 1$, shows that all residue classes of $x$ except one are constant, and up to a shift coincide with those of $y$. Therefore, $x$ is in the orbit closure of $y$ and is (uniformly) recurrent.
    
\end{proof}

We now examine the case where $y$ is a simple circle rotation coding sequence. 

%By \Cref{mainthm:recnotmin}, any sequence with non-minimal pattern Sturmian orbit closure must be nonrecurrent. We can restrict even this case significantly.

\begin{lemma}\label{lem:two_coding_same_tail}
Let $I, J$ be two nonempty intervals on $\T$ and $\alpha$ be an irrational number. If there exist arbitrarily long intervals $F \subseteq \N_0$ such that $1_I(n \alpha) = 1_J(n \alpha)$ for $n \in F$ then $I = J$ except possibly at the endpoints.
%Let $I, J$ be two nonempty intervals on $\T$ and $\alpha$ be an irrational number. If there exists $N \in \N$ such that $1_I(n \alpha) = 1_J(n \alpha)$ for $n \geq N$ then $I = J$ probably except at the endpoints.
\end{lemma}
%\Anote{To deal with the situation that $x$ is not in the form $wy$ where $y$ is infinite minimal pattern Sturmian and $w$ is finite, we need to strengthen \cref{lem:two_coding_same_tail} to following: 

%Let $I, J$ be two nonempty intervals on $\T$ and $\alpha$ be an irrational number. If there exist arbitrarily long interval $F \subseteq \N$ such that $1_I(n \alpha) = 1_J(n \alpha)$ for $n \in F$ then $I = J$ probably except at the endpoints.
%}
\begin{proof}
% Suppose $x = (1_{I}(n \alpha))_{n \in \Z}$ and $y = (1_J(n \beta))_{n \in \Z}$ where $I, J \subseteq \T$ are intervals and $\alpha, \beta$ are irrational numbers. 

% By the uniform distribution of irrational rotations, 
% \[
%     |I| = \lim_{N \to \infty} \frac{|x^{-1}(1) \cap [1, N]|}{N}
% \]
% and
% \[
%      |J| = \lim_{N \to \infty} \frac{|y^{-1}(1) \cap [1, N]|}{N}.
% \]
% It follows that $|I| = |J|$.

% Now we show $\alpha = \beta$. Without loss of generality, assume $|I| = |J| \leq 1/2$. The sequence $(x(kn))_{n \in \N}$ has infinitely many pairs of consecutive $1$ if and only if $k \alpha \in (0, |I|) \cup (1 - |I|, 1)$ (or the same union including some endpoints depending on $I$). Similarly, $(y(kn))_{n \in \N}$ has infinitely many pairs of consecutive $1$ if and only if $k \beta \in (0, |J|) \cup (1 - |J|, 1)$. As $|I| = |J|$, for all $k \in \Z$,
% \[
%     k \alpha \in (0, |I|) \cup (1 - |I|, 1) \Leftrightarrow k \beta \in (0, |I|) \cup (1 - |I|, 1). 
% \]
% From this we can deduce that $\alpha = \beta$. (It is easy to see if $\alpha, \beta$ are rationally independent. If $\alpha, \beta$ are dependent, it can be seen by drawing the graph $y = (\beta/\alpha) x$ on $\T \times \T$.) \Anote{Annoying issue arises when $|I| = |J| = 1/2$. Will fix later, maybe with a different idea.}

If the conclusion does not hold, then $(I \setminus J) \cup (J \setminus I)$ contains a nonempty interval. Since $(n \alpha)_{n \in \N_0}$ is dense on $\T$, and the system $(\T, n\alpha)$ is minimal, the set of return times to $(I \setminus J) \cup (J \setminus I)$ is syndetic. Thus, there exists some $r$ so that if $|F| > r$, then there is some $n \in F$ such that $n \alpha \in (I \setminus J) \cup (J \setminus I)$.   For this $n$ we will have $1_I(n \alpha) \neq 1_J(n \alpha)$, a contradiction.
%there exists (arbitrarily large) $n$ such that $n \alpha \in (I \setminus J) \cup (J \setminus I)$. For this $n$ we will have $1_I(n \alpha) \neq 1_J(n \alpha)$, a contradiction.
%\Cnote{For the new conclusion, I think all we need to add is that since rotation by alpha is minimal on a compact set, the return time to the interval is syndetic so there cannot be arbitrarily long stretches where $1_J(n\alpha) = 1_I(n\alpha)$, right?}
\end{proof}

\begin{proposition}\label{lem:glue_up}
Let $x \in \{0, 1\}^{\N_0}$ be a pattern Sturmian sequence. Let $I$ be a nonempty, proper interval of $\T$ and $\alpha$ be an irrational number. Suppose there exist arbitrarily long intervals $F \subseteq \N_0$ such that $x(n) = 1_I(n \alpha)$ for all $n \in F$. Then $x(n) = 1_J(n \alpha)$ for all $n \in \N_0$ where $J = I$ except possibly at the end points.
%Suppose $x \in \{0, 1\}^{\N}$ is a pattern Sturmian sequence such that $x(n) = 1_I(n \alpha)$ for $n \geq N$, where $I$ is a nonempty, proper interval of $\T$ and $\alpha$ is irrational. Then $x(n) = 1_J(n \alpha)$ for all $n \in \N$ where $J = I$ probably except at the end points.
\end{proposition}
%\Anote{Similarly, we need to change \cref{lem:glue_up} to following:

%Let $x \in \{0, 1\}^{\N}$ be a pattern Sturmian sequence. Let $I$ be a nonempty, proper interval of $\T$ and $\alpha$ be an irrational number. Suppose there exists arbitrarily long interval $F \subseteq \N$ such that $x(n) = 1_I(n \alpha)$ for all $n \in F$. Then $x(n) = 1_J(n \alpha)$ for all $n \in \N$ where $J = I$ probably except at the end points.
%}
\begin{proof}
Let $y(n) = 1_I(n \alpha)$ for all $n \in \N_0$.
By \cref{lem:sturmian_no_constant_words}, there exists $k$ such that for all sufficiently large $n$, the window $\tau = \{0, k, 2k, \ldots, (n-1)k\}$ satisfies $|L_y(\tau)| = 2n$. By uniform recurrence of $y$, there exists $N$ so that every $N$-letter subword of $y$ contains all words in $L_y(\tau)$. Since $(x(n))_{n \in F} = (y(n))_{n \in F}$ for some interval of length at least $N$, $L_y(\tau) \subseteq L_x(\tau)$. Since $x$ is pattern Sturmian, $L_x(\tau) = L_y(\tau)$.  %  \rp{(is this from a previous version? All we know here is that $x$ contains arbit long words of $y$. That should still be good enough by uniform recurrence of $y$ though)}\Cnote{yep, it's from the previous version, I think I've fixed it now}, %we have $L_x(\tau) = L_y(\tau)$ (and is equal to $L_{(y(kn))_{n \in \N_0}}(\{0, 1, \ldots, n-1\})$. Because $n$ is arbitrary 
%\rp{I don't see what this means. Is it about arbitrariness of the residue class mod $n$?}\Cnote{It was referencing arb length  of the window, but I changed to for all n at the beginning and took it out to make it clearer, if this looks better}, 
Thus the languages of $(x(kn))_{n \in \N_0}$ and $(y(kn))_{n \in \N_0}$ are the same. It follows that $(x(kn))_{n \in \N_0}$ belongs to the orbit closure of $(y(kn))_{n \in \N_0}$, i.e. $x(kn) = 1_{x_0 + J_0}(kn\alpha)$ for all $n \in \N_0$, where $x _0 \in \T$ and the interval $J_0$ is equal to $I$ except possibly at the end points. By a similar argument, for $j = 0, 1, \ldots k - 1$, 
\[
    x(kn + j) = 1_{x_j + J_j}(kn \alpha) \text{ for all } n \geq 0,
\]
where $x_j \in \T$ and $J_i = I$ except possibly at the endpoints. 
%In particular, all the sequences $(x(kn + j))_{n \in \Z}$ are codings of $k \alpha$ by two interals.

%\Cnote{Here I think the proof is exactly the same except to modify 'tails of the sequence' to for arb long intervals, but I think everything else still works}
%It follows from our assumption that the tails of the sequence
It follows from our assumption that for arbitrarily long intervals $F \subseteq \N_0$,
\[
    (x(kn + j))_{n \in \N_0} = (1_{x_j + J_j}(kn \alpha))_{n \in \N_0}
\]
and the sequence
\[
    (y(kn + j))_{n \in \N_0} = (1_I((kn + j)\alpha))_{n \in \N_0} = (1_{I- j \alpha}(kn \alpha))_{n \in \N_0}
\]
are the same. By \cref{lem:two_coding_same_tail}, $x_j + J_j = I - j \alpha$ except possibly at the end points. Letting $I_j = x_j + J_j + j \alpha$, we have $I_j = I$ except possibly at the endpoints and
\[
    (x(kn + j))_{n \in \N_0} = (1_{I_j}((kn + j)\alpha))_{n \in \N_0}.
\]
Since $(n \alpha)$ lands at each endpoint of $I$ at most once, we can modify $I$ at the endpoints to create an interval $J$ such that $x(n) = 1_J(n \alpha)$ for all $n \in \N_0$.
\end{proof}

\begin{proposition}\label{prop:nonrecurrent_pattern_Sturmian}
If $x \in \{0, 1\}^{\N_0}$ is a nonrecurrent, pattern Sturmian sequence whose orbit closure contains an infinite minimal subsystem, then $x(n) = 1_{I}(n \alpha)$ for all $n \in \N_0$, where $\alpha$ is irrational and $I$ is an interval of $\T$ of the form $(k_1 \alpha, k_2 \alpha)$ or $[k_1 \alpha, k_2 \alpha] \bmod 1$ for some $k_1 \neq k_2 \in \N_0$.
\end{proposition}
\begin{proof}
%\rp{proof looks good to me now!}
%\rp{I was very worried about this proof, but now I think maybe it's just outdated. Shouldn't \Cref{lem:glue_up} allow us to prove this without breaking into cases? Namely, $y$ can't be nearly simple Toeplitz, so it's a coding sequence, and then $x$ contains arbitrarily long words from $y$, so we can use \Cref{lem:glue_up}. Hopefully this is right and then this just needs to be rewritten.}
%\Cnote{I think I have fixed this now}

Let $x$ be a nonrecurrent, pattern Sturmian sequence whose orbit closure contains an infinite minimal subsystem $Y$.  Since $x$ is pattern Sturmian, so is $Y$, and so by \cref{mainthm:sturmian}, $Y$ is the orbit closure of a simple circle rotation coding sequence or a nearly simple Toeplitz sequence. By \cref{thm:contain_minimal_toeplitz_cannot}, $Y$ cannot be the orbit closure of a nearly simple Toeplitz sequence, and so $Y$ must be the orbit closure of a simple circle rotation coding sequence. Thus for arbitrarily long intervals $F \subseteq \N_0$, $(x(n))_{n \in F} = (y(n))_{n \in F}$ for some $y$ a simple circle rotation coding sequence.  It then follows from \cref{lem:glue_up} that $x(n) = 1_{I}(n \alpha)$ for all $n \in \N_0$, where $\alpha$ is irrational and $I$ is an interval of $\T$.

If $I$ has the form $[a, b)$ or $(a, b]$, then $x$ is recurrent and this contradicts our hypothesis.
If one of the endpoints of $I$ is not in $\{k \alpha: k \in \N_0\}$, then we can change $I$ to an interval of the form $[a, b)$ and the sequence $x$ stays the same. In this case, $x$ is recurrent and again this contradicts our hypothesis. Thus $I = (k_1 \alpha, k_2 \alpha)$ or $[k_1 \alpha, k_2 \alpha]$ for some $k_1, k_2 \in \N_0$ and we are done.
\end{proof}

% \begin{proposition}\label{lem:nonrecurrent_infinite_minimal}
% If $x$ is nonperiodic and not recurrent with pattern Sturmian orbit closure $X$, then every minimal subshift contained in $X$ is finite.
% \end{proposition}

% \begin{proof}
% LATER. Get roughly $3n$ word for regular word complexity; $n$ via word $w$ appearing only once, $n$ via infinite minimal subshift (which can't contain $w$), and $n$ via transitioning out of $M$.
% \end{proof}

% The following is immediate.

% \begin{corollary}
% If $X$ is pattern Sturmian and not minimal, then every minimal subshift of $X$ is finite.
% \end{corollary}

%\section{Nonminimal pattern-Sturmian sequences}

% A sequence $\alpha \in \{0, 1\}^{\N}$ is called a \emph{sparse sequence} if the distance between consecutive $0$'s or $1$'s goes to infinity.

% \begin{proposition}
% If $\alpha \in \{0, 1\}^{\N}$ is a nonrecurrent, non-sparse sequence, then $p_{\alpha}^*(k) > 2k$ for large $k$.    
% \end{proposition}
% \begin{proof}
%     ...
% \end{proof}

% \begin{proposition}\label{lem:nonrecurrent_infinite_minimal}
% If $\alpha$ is a nonrecurrent sequence whose orbit closure contains an infinite minimal subsystem, then $\alpha$ is not pattern Sturmian.
% \end{proposition}
% \begin{proof}
% \Anote{Need proof...}
% \end{proof}

% \rp{I'm a little lost about these results; does Proposition 4.1 replace 4.3? Do we still need a proof for 4.3, or is the current 4.1 outline the plan? We can talk about this in person maybe.}

Finally, we deal with the case where all possible $m$ are periodic; the eventual goal is to show that $m$ is constant and $x$ is almost constant.

\begin{lemma}\label{lem:long_0_and_1_not_pattern_Sturmian}
If $x \in \{0, 1\}^{\N_0}$ has arbitrarily long blocks of consecutive $0$s and arbitrarily long blocks of consecutive $1$s, then $x$ is not pattern Sturmian. 
\end{lemma}
\begin{proof}
%I rewrote this a bit, it's a bit shorter and I don't think it really needed to be a contradiction proof. But the old proof is there, replace as needed.
Assume that $x$ has arbitrarily long blocks of $0$s and $1$s, and for any $n$, let $\tau$ be the window 
$\{0, 1, \ldots, n-1\}$. By considering shifts of $\tau$ whose right edge lies within a block of $0$s of length at least $n$ in $x$, we see that $L_x(\tau)$ contains $0^n$ and, for every $0 < i < n$, a word ending with $10^i$. A similar argument using a long block of $1$s shows that $L_x(\tau)$ contains $1^n$ and, for every $0 < i < n$, a word ending with $01^i$. All of these words are distinct, yielding $2n$ words in $L_x(\tau)$. Now, choose any $\ell$ for which $x$ contains a block of $0$s of length exactly $\ell$, meaning that $x$ contains $10^\ell 1$. Since $x$ contains a block of $0$s of length greater than $\ell$, $x$ also contains $0^{\ell + 1} 1$. 

But both of these words end with $01$, meaning that for $n = \ell + 2$, $L_x(\tau)$ is strictly larger than the previous set of $2n$ words, so $p^*_x(n) > 2n$ and $x$ is not pattern Sturmian.
%Then for $n = \ell + 2$, $L_x(\tau)$ contains a word not in the previous list
%\begin{align*}
%\label{eq:list_n_words}
 %   &x_1 x_2 x_3 \ldots x_{n-2} 1 0\\
  %  &x_2 x_3 \ldots x_{n-2} 100\\
  %  &x_3 \ldots x_{n-2} 1000\\
  %  &\ldots\\
  %  &000 \ldots 000
%\end{align*}
%and when we slide $\tau$ into a long block of $1$s, we pick up
%\begin{align*}
 %   &y_1 y_2 y_3 \ldots y_{n-2} 0 1\\
  %  &y_2 y_3 \ldots y_{n-2} 011\\
 %   &y_3 \ldots y_{n-2} 0111\\
 %   &\ldots\\
 %   &111 \ldots 111,
%\end{align*}
%where $x_i, y_i \in \{0, 1\}$.
%There are $2n$ of them in total and it is easy to see that all of them are distinct. Since $p_{x}^*(n) %\leq 2n$, there is no other $n$ length word in $x$. 

%Now when we slide $\tau$ exiting a long block of $1$s, we will pick up the word $1 \ldots 1 0$. This word must belong to the list above and in fact it must be the first word in the list. Therefore, $x_1 = \ldots = x_{n-2} = 1$. By a similar argument, $y_1 = \ldots = y_{n-2} = 0$.

%If we choose $n = \ell + 2$ where $\ell$ is the length of a block of $0$s (and so this block has $\ell$ $0$s surrounded by two $1$s). Then $x$ contains the word $100 \ldots 001$ of length $n$ which does not belong to the list above, contradicting our original assumption and completing the proof.
\end{proof}

\begin{proposition}\label{prop:no-two-orbit}
If the orbit closure of $x \in \{0, 1\}^{\N_0}$ has two distinct periodic orbits, then $x$ is not pattern Sturmian. 
\end{proposition}
\begin{proof} 
Since $x$ is not eventually periodic, it suffices to show that $p_{x}^*(n) \geq 2n + 1$ for some $n$.
Suppose that the orbit closure of $x$ contains periodic sequences $w_1 w_1 \ldots$ and $w_2 w_2 \ldots$ with distinct orbits. Then %, meaning that %Let $w_1, w_2$ be the two finite words that generates disjoint periodic orbits such that 
$x$ contains arbitrarily long blocks of the forms $w_1 w_1 \ldots w_1$ and $w_2 w_2 \ldots w_2$.

Let $k$ be a common multiple of the lengths of $w_1$ and $w_2$. Then for every $j \in \{0, \ldots, k - 1\}$, the sequence $(x(kn + j))_{n \in \N_0}$ contains arbitrarily long blocks of $0$s or $1$s. 

We claim that there exists $j$ such that $(x(kn + j))_{n \in \N_0}$ contains arbitrarily long blocks of $1$s and arbitrarily long blocks of $0$s. Assume this is not the case. Let $M$ be such that for all $j$, the longest $0$ or $1$-block of $(x(kn + j))_{n \in \N_0}$ is bounded above by $M$. For $i \in \{1, 2\}$, let $m_i$ be such that the word $x([k m_i, k (m_i + M)])$ is a subword of a block of the form $w_i w_i \ldots w_i$. Because for all $j = 0, \ldots, k - 1$, the size of window $\{km_1 + j, \ldots, k(m_1 + M) + j\}$ is $M + 1$ (in particular greater than $M$), by our assumption, 
\[
    x([k m_1 + j, \ldots, k (m_1 + M) + j]) = x([k m_2 + j, \ldots, k (m_2 + M) + j]).
\]
In particular, for all $j = 0, \ldots, k - 1$,
\[
    x(k m_1 + j) = x(k m_2 + j).
\]
It follows that $w_1, w_2$ generate the same periodic orbit (i.e. one is a rotation of the other) and this is a contradiction.

Fixing $j$ found in the previous paragraph, then the subsequence $(\beta(n))_{n \in \N_0} = (x(kn + j))_{n \in \N_0}$ has arbitrarily long blocks of consecutive $0$s and consecutive $1$s. By \cref{lem:long_0_and_1_not_pattern_Sturmian}, $p_{\beta}^*(n) \geq 2n + 1$ for some $n$ and our lemma follows since $p_{x}^*(n) \geq p_{\beta}^*(n)$.
\end{proof}

\begin{proposition}\label{prop:w-constant}
Let $x \in \{0, 1\}^{\N_0}$ be such that $x$ differs from the infinite sequence $w w w \ldots$ on a set of Banach density zero for some finite word $w$. If $w$ is not a constant word, then $x$ is not pattern Sturmian. 
\end{proposition}
\begin{proof}
Assume that $x$ and $w$ are as in the statement, and assume that $w$ contains both $0$ and $1$. 
Let $k = |w|$, the length of $w$. Consider the sequences $(x^{(j)}(n))_{n \in \N_0} = (x(kn + j))_{n \in \N_0}$ for $j \in \{0, \ldots, k - 1\}$. For each $j$, $x^{(j)}$ contains arbitrarily long blocks of consecutive $0$s or $1$s or both. However, the last possibility is ruled out by \cref{lem:long_0_and_1_not_pattern_Sturmian}, since $p^*_x(n) \geq p^*_{x^{(j)}}(n)$. 
Since $w$ contains both $0$ and $1$, there are $j_0, j_1 \in \{0, \ldots, k-1\}$ that $x^{(j_0)}$ contains arbitrarily long blocks of $0$s and $x^{(j_1)}$ contains arbitrarily long blocks of $1$s. Because $x$ is nonperiodic, we can choose $j_0, j_1$ so that at least one of $x^{(j_0)}, x^{(j_1)}$ is nonperiodic. Without loss of generality (and by switching $0$ and $1$ if necessary, which does not change any hypothesis or conclusion of the proposition), assume $x^{(j_0)}$ is nonperiodic. 

We will show that there exists a $3$-window $\tau$ such that 
\[
    |L_{x^{(j_0)}}(\tau) \cup L_{x^{(j_1)}}(\tau)| \geq 7.
\]
It then easily follows that $p_{x}^*(3) \geq 7$, implying that $x$ is not pattern Sturmian.

For any $3$-window $\tau$, $L_{\beta_0}(\tau)$ always contains $000$ and $L_{\beta_1}(\tau)$ always contains $111$. Moreover, if $n$ is the location of an $1$ before entering a long block of $0$s, then $x(n + \tau) = 100$. Thus $L_{\beta_0}(\tau)$ always contains $100$, and a similar argument shows it always contains $001$.

Let $S = \{s_1 < s_2 < \ldots\}$ be the set of locations where $1$ appears in $x^{(j_0)}$ and define $g_n = s_{n+1} - s_n$. The set $S$ is not syndetic because $d^*(S) = 0$ and so there exists $n$ such that 
\begin{equation}\label{eq:condition_s_n}
    g_n < g_{n+1}.
\end{equation}
Choose the window
\begin{equation}\label{eq:choose_tau}
    \tau = \{0, g_n, g_{n+1}\}.
\end{equation}
Then 
\[
    x^{(j_0)}(s_n + \tau) = x^{(j_0)}(\{s_n, s_{n+1}, s_n + g_{n+1}\}) = 110
\]
because $s_{n+1} < s_n + g_{n+1} < s_{n+2}$. Furthermore,
\[
    x^{(j_0)}(s_{n+1} + \tau) = x^{(j_0)}(\{s_{n+1}, s_{n+1} + g_n, s_{n+2}]\}) = 101
\]
because $s_{n+1} < s_{n+1} + g_n < s_{n+2}$.

So far with the window $\tau$ in \eqref{eq:choose_tau}, we have shown $L_{x^{(j_0)}}(\tau) \cup L_{x^{(j_1)}}(\tau)$ contains $6$ words $000, 001, 100, 111, 110, 101$. It remains to find $n$ so that we pick up the extra word $010$ when sliding the window $\tau$ along $x^{(j_0)}$. For contradiction, assume there exists no such $n$. 

Fix an $n_0$ such that $g_{n_0+1} > g_{n_0}$. We claim that for all $\ell$, there exists $k \geq 2$ such that $g_{k-1} \leq g_{n_0}$ and $g_{k} > \ell$. For sufficiently large $k$ (more specifically, if $s_k > g_{n_0}$), we have
\begin{equation}\label{eq:layout}
    x^{(j_0)}(s_k - g_n + \tau) = x^{(j_0)}(\{s_k - g_n, s_k, s_k + (g_{n+1} - g_{n})\}) = u1v
\end{equation}
where $u, v \in \{0, 1\}$. By our contradiction assumption, $u = 1$ or $v = 1$. If $s_k$ is the location of a $1$ right before a very large block of $0$s (say the length of this block is larger than $\ell$), then $x^{(j_0)}(s_k + (g_{n_0+1} - g_{n_0})) = 0$ and so this forces $x^{(j_0)}(s_k - g_{n_0}) = 1$. Since $s_{k-1}$ is the location of the last $1$ before $s_k$, we have 
\[
    g_{k-1} = s_k - s_{k-1} \leq s_k - (s_k - g_{n_0}) = g_{n_0}.
\]
On the other hand, since the digit $1$ at the location $s_k$ is followed by an $\ell$-block of $0$s, $g_k = s_{k+1} - s_k > \ell$. Our claim follows.

Let $\ell$ be arbitrary and fix $k$ so that $g_{k-1} \leq g_{n_0}$ and $g_k > \ell g_{n_0} \geq \ell g_{k-1}$.
We will show that $x^{(j_0)}$ contains a sequence of $\ell + 1$ $1$ symbols where the gaps between consecutive $1$s are bounded by $g_{n_0}$. Since $\ell$ is arbitrary this will imply that the upper Banach density of $1$ in $x^{(j_0)}$ is at least $1/g_{n_0}$ and this contradicts the assumption that this density is zero. Thus we will be done.

Let $s_t$ be the location of an $1$ before a block of consecutive $0$s of length greater than $g_k - g_{k-1}$.
Looking at \eqref{eq:layout}, because 
\[
    x^{(j_0)}(s_t) = 1 \text{ and } x^{(j_0)}(s_t + g_{k} - g_{k-1}) = 0,
\]
it must be that $x^{(j_0)}(s_t - g_{k-1}) = 1$. Now let $s_t - g_{k-1}$ plays the role of $s_t$. We then have 
\[
    x^{(j_0)}(s_t - g_{k-1}) = 1 \text{ and } x^{(j_0)}(s_t - g_{k-1} + (g_k - g_{k-1}))= x^{(j_0)}(s_t + g_k - 2 g_{k-1}) = 0
\]
force $x^{(j_0)}(s_t - 2 g_{k-1}) = 1$. Continuing in this way, we obtain $1$ at the following $\ell + 1$ locations in $x^{(j_0)}$
\[
    s_t, s_t - g_{k-1}, \ldots, s_t - \ell g_{k-1}.
\]
At each step, we use the fact that $g_k > \ell g_{k-1}$ to make sure $s_t + g_k - \ell g_{k-1} > s_t$ and so $x^{(j_0)}(s_t + g_k - \ell g_{k_1}) = 0$. We have arrived at a contradiction, and so $x$ is not pattern Sturmian, completing the proof.
\end{proof}

\begin{proof}[Proof of \cref{mainthm:sturmian-nonrecurrent}]
%\cref{mainthm:sturmian-nonrecurrent} follows from Propositions \ref{prop:no-two-orbit} and \ref{prop:w-constant}.

Let $x$ be a nonrecurrent, pattern Sturmian sequence and let $X$ be the orbit closure of $x$. Let $y$ be a uniformly recurrent sequence in $X$. Then $y$ is periodic or pattern Sturmian. If $y$ is pattern Sturmian, by \cref{mainthm:sturmian}, $y$ is either in a nearly simple Toeplitz subshift (in which case it can be taken to be itself nearly simple Toeplitz) or a simple circle rotation coding sequence. According to  \cref{thm:contain_minimal_toeplitz_cannot}, the former case is impossible. If $y$ is a simple circle rotation coding sequence, then by \cref{prop:nonrecurrent_pattern_Sturmian}, $x$ is a (nonrecurrent) simple circle rotation coding sequence %(more precisely, $x$ is an ``almost Sturmian'') a
and we are done. 

Suppose $y$ is periodic and let $Y$ be the finite orbit of $y$ (which is already closed). By \cref{prop:no-two-orbit}, $X$ does not contain any other periodic subsystem. We can further assume that $X$ does not contain any nonperiodic minimal subsystems (otherwise we return to the previous already-treated cases). 

Now \Cref{mainthm:recnotmin} implies that every recurrent point in $X$ belongs to $Y$. It follows that $X$ is uniquely ergodic with the unique measure being the uniform probability measure $\mu$ on $Y$. (This is because due to Poincar\'e's recurrence theorem, the support of any invariant measure contains a recurrent point. Thus, if there were a second invariant measure, the set of recurrent points would be larger than $Y$.) 
By unique ergodicity, $x$ is a generic point for $\mu$, which implies that it differs from some point of $Y$, which must be of the form $.www\ldots$, on a set of zero Banach density. \cref{prop:w-constant} implies that $w$ is a constant word and so $x$ is almost constant. 
\end{proof}

We still do not know exactly which almost constant sequences are pattern Sturmian (recall that $x = 1_S$ for $S = \{s_1, s_2, \ldots\}$ with $s_{k+1} > 2s_k$ is pattern Sturmian (\cite{Kamae-Zamboni-sequence_entropy}), but any sequence starting with $00001011100$ is not), leading to the following question.

\begin{question}\label{q:almcon}
What else can we say about almost constant pattern Sturmian sequences? For instance, are there stronger senses than upper Banach density in which the deviations from constancy are `small'?
\end{question}

%\rp{I thought of two more questions that are quite natural if we want to make a questions section. The first is: can the lower bound of Theorem D be improved? We prove $O(n \ln n)$, and it's an open question if there are null examples. The second is: Can one characterize which $m$-hole Toeplitzes have linear max pattern complexity (as a sort of converse to Theorem C)? I believe this question is open even for $m = 1$, though the Gjini et al paper may yield an answer if studied carefully.}

%\Anote{Need at least two more questions....} \rp{I'd settle for one more, but I agree one is too few. We could ask about minimal pattern complexity of 2-hole Toeplitz sequences, or about the minimal pattern complexity of a $2$-torus rotation.}

\bibliographystyle{abbrv}
\bibliography{refs}

\end{document}